\documentclass[11pt,leqno,a4paper]{article}

\makeatletter
  \renewcommand{\theequation}{%
         \thesection.\arabic{equation}}
  \@addtoreset{equation}{section}
\makeatother

\usepackage{upref,amsxtra,amscd,amsopn,amsbsy,amstext,amssymb,amsmath,amsthm}
\usepackage{bm}
\usepackage{xcolor,esint}
\usepackage{hyperref}

\newtheorem{thm}{Theorem}[section]
\newtheorem{lem}[thm]{Lemma}
\newtheorem{prop}[thm]{Proposition}
\newtheorem{cor}[thm]{Corollary}
\newtheorem{rem}[thm]{Remark}

\newtheorem{dfn}[thm]{Definition}

\newcommand{\bt}{\boldsymbol{\theta}}
\newcommand{\pd}{\partial}
\newcommand{\gm}{\gamma}
\newcommand{\Gm}{\Gamma}

\newcommand{\R}{\mathbb{R}}
\newcommand{\Z}{\mathbb{Z}}
\newcommand{\N}{\mathbb{N}}

\newcommand{\vk}{\kappa}
\newcommand{\vp}{\varphi}

\newcommand{\V}{\boldsymbol{V}}
\newcommand{\ve}{\varepsilon}

\usepackage{color}

\allowdisplaybreaks

\begin{document}

\title{\bf A second order gradient flow  of  $p$-elastic planar networks}
\author{ 
{\sc Matteo Novaga}  
\thanks{Dipartimento di Matematica, Universit\`a di Pisa, Pisa, Italy, \url{matteo.novaga@unipi.it}} 
\and
{\sc Paola Pozzi}  
\thanks{Fakult\"at f\"ur Mathematik, Universit\"at Duisburg-Essen, Essen, Germany, \url{paola.pozzi@uni-due.de}}
}
\date{\today}

\maketitle

\begin{abstract}
\noindent We consider a second order gradient flow of the $p$-elastic energy for a planar theta-network of three curves 
with fixed lengths. 
We construct a weak solution of the flow by means of an implicit variational scheme.
We show long-time existence of the evolution and convergence to a critical point of the energy.
\end{abstract}

{\small\textbf{Keywords:} elastic flow of networks,  minimizing movements, long-time existence

}




\section{Introduction}
In this paper we consider a network composed of three inextensible planar curves.
Each  curve $\gm_i=\gm_i(s)$ of fixed length $L_i>0$, $i=1,2,3$, is parametrized by arc-length $s$ over the domain $\bar{I}_i=[0,L_i]$. Without loss of generality we may assume that
\begin{align}\label{lengths}
 0 < L_{3} \leq \min \{L_{2}, L_{1} \}.
\end{align}
Let $T^i=T^i(s)=\gamma_i'(s)$ denote the unit tangent of the curve $\gamma_i$. It is well known  that   a planar curve is uniquely determined by its tangent indicatrix $T^{i}$, up to rotation and translation.
Omitting for simplicity the indices of the curves, we recall the formulas $T'= \vec{\kappa}=\vk N$, $N'=-\vk T$, as well as $\theta'(s)= \vk(s)$, where  $T=(\cos \theta, \sin \theta)$. The map $\theta: I \to \R$ is called the indicatrix of the curve $\gamma$. 

We shall consider a {\it theta-network} $\Gamma =\{ \gm_{1}, \gm_{2}, \gm_{3} \}$, 
where the three curves satisfy
\begin{align*}
\gamma_{1}(0)&=\gamma_{2}(0)=\gamma_{3}(0),\\
\gamma_{1}(L_{1}) &= \gm_{2}(L_{2})= \gm_{3}(L_{3}).
\end{align*}
Without loss of generality we shall assume that the first triple point is  placed at the origin, that is, $\gamma_{i}(0)=0$, for $i=1,2,3$. From the concurrency conditions above it follows immediately that
\begin{align}\label{constraintN}
\int_{I_{1}} T^{1}(s) ds= \int_{I_{2}} T^{2}(s) ds=\int_{I_{3}} T^{3}(s) ds.
\end{align}

Letting $p\in (1,+\infty)$, the $p$-elastic energy of the network is defined as
\begin{align*}
E_{p}(\Gamma)= \sum_{i=1}^{3} E_{p}(\gamma_{i}),
\end{align*}
where
\begin{align*}
E_{p}(\gamma_{i}):=\frac{1}{p} \int_{I_{i}}|\vec{\vk}_{i}|^{p} ds = \frac{1}{p} \int_{I_{i}} |\partial_{s} T^{i}|^{p} ds=:F_{p}(T^{i}).
\end{align*}

Minimizers for the elastic energy (i.e. with $p=2$) plus an additional term penalizing the growth of the length of the curves  have been investigated in \cite{DNP19}, where
an angle condition at the triple junctions has been imposed in order to avoid the collapse of a minimizing sequence to a point. Here the situation is different, because the length of each curve is fixed. In particular it is not  necessary to impose the angle condition at the triple junctions.

Here we consider the evolution of  the network $\Gamma$ via a {\it second order} gradient flow first introduced by Y. Wen 
in \cite{Wen} (see also \cite{Lin15,OPW18,LinKai18}). More precisely we will consider the $L^2$-gradient flow of the energy
$$F_{p}(\Gamma):=\sum_{i=1}^{3}F_{p}(T^{i}),$$ when expressed in terms of the angles corresponding to the tangent vectors.
This gives rise to a second order parabolic system.

We shall express the energy $F_{p}(\Gamma)$ and the corresponding gradient flow by means of 
the three scalar maps $\theta^{i}: I_{i} \to \R$ such that $T^i=(\cos \theta^i, \sin \theta^i)$.
Let us now state our main existence results.
We let
\begin{align*}
H := \Big \{ \boldsymbol{\theta}=&(\theta^{1}, \theta^{2}, \theta^{3}) \in W^{1,p}(0,L_{1}) \times W^{1,p}(0,L_{2}) \times W^{1,p}(0,L_{3}) \, | \,\\
&\int_{I_{1}} (\cos \theta^{1}, \sin \theta^{1}) ds = \int_{I_{2}} (\cos \theta^{2}, \sin \theta^{2}) ds =\int_{I_{3}} (\cos \theta^{3}, \sin \theta^{3}) ds \Big\}.
\end{align*}

\begin{thm}\label{thm:existence}
Let $\bt_{0} \in H$ and let $T>0$.   
Assume that the lengths of the three curves are such that 
\begin{align}\label{cond1}
 L_{3} < \min \{ L_{1}, L_{2}\} .
 \end{align}
Then, there exist functions 
$\bt= (\theta^{1}, \theta^{2}, \theta^{3})$, with $\theta^{j} \in L^{\infty}(0,T; W^{1,p}(I_{j})) \cap H^{1}(0,T; L^{2}(I_{j}))$,
and Lagrange multipliers $\lambda^{1}, \lambda^{2}, \mu^{1}, \mu^{2} \in L^{2}(0,T)$ such that
the following properties hold:

(i) for any $\boldsymbol{\varphi}=(\vp^{1},\vp^{2}, \vp^{3})$ with $\varphi^{j} \in L^{\infty} (0,T; W^{1,p}(I_{j}))$, $j=1,2,3,$  there holds
\begin{align}\nonumber
0&=\sum_{j=1}^{3} \int_{0}^{T}\int_{I_{j}}\partial_{t} \theta^{j}\,  \varphi^{j} ds   dt
+ \sum_{j=1}^{3} \int_{0}^{T}\int_{I_{j}} | \theta^{j}_{s} |^{p-2}  \theta^{j}_{s}  \cdot  \vp_{s} \, ds dt \\\label{eqtest}
& \qquad-\int_{0}^{T}({\lambda}^{1} - \mu^{1})\int_{I_{1}} \sin (\theta^{1}) \, \varphi^{1} ds  dt
+ \int_{0}^{T}(\lambda^{2} -\mu^{2})\int_{I_{1}} \cos (\theta^{1} )\,  \varphi^{1} ds dt \\\nonumber
& \qquad +\int_{0}^{T}\lambda^{1} \int_{I_{2}} \sin (\theta^{2}) \, \vp^{2} ds  dt -\int_{0}^{T}\lambda^{2} \int_{I_{2}} \cos (\theta^{2}) \, \vp^{2} ds dt \\\nonumber
& \qquad -\int_{0}^{T}\mu^{1} \int_{I_{3}} \sin (\theta^{3}) \, \vp^{3} ds dt  +\int_{0}^{T}\mu^{2}\int_{I_{3}} \cos (\theta^{3}) \, \vp^{3} ds dt\, ;
\end{align}

(ii) the maps $ |\partial_{s} \theta^{j}|^{p-2}  \partial_{s} \theta^{j}$ belong to 
$L^{\infty}(0,T; L^{\frac{p}{p-1}}(I_{j}))\cap L^{2}(0,T; H^{1}(I_{j}))$, $j=1,2,3$,
and satisfy
\begin{align}\label{E1}
(|\partial_{s} \theta^{1}|^{p-2}\partial_{s} \theta^{1})_{s}&=\theta^1_{t} - (\lambda^{1}-\mu^{1}) \sin \theta^{1}  
+ (\lambda^{2}-\mu^{2}) \cos\theta^{1}, \\\label{E2}
 (|\partial_{s} \theta^{2}|^{p-2}\partial_{s} \theta^{2})_{s}&=\theta^2_{t} + \lambda^{1} \sin \theta^{2}  
 - \lambda^{2} \cos \theta^{2},\\\label{E3}
 (|\partial_{s} \theta^{3}|^{p-2}\partial_{s} \theta^{3})_{s}&=\theta^3_{t} -\mu^{1} \sin \theta^{3}  
 +\mu^{2} \cos \theta^{3}, \\\label{E4}
 \theta^{j}_{s} (0,t)&=\theta^{j}_{s} (L_{j},t) =0,\ \text{for $j=1,2,3$ and for a.e. $t \in (0,T)$;}
\end{align}

(iii) for all $t \in [0,T]$, there holds
\begin{align}\label{constraintlimit}
\int_{I_{1}} (\cos \theta^{1}, \sin \theta^{1}) ds = \int_{I_{2}} (\cos \theta^{2}, \sin \theta^{2}) ds=\int_{I_{3}} (\cos \theta^{3}, \sin \theta^{3}) ds.
\end{align}
\end{thm}

Notice that the time $T>0$ can be chosen arbitrarily, so that the weak solutions 
$\bt$ and the Lagrange multipliers $\vec{\lambda}=(\lambda^{1}, \lambda^{2}),\,\vec{\mu}=(\mu^{1}, \mu^{2})$ can be defined globally on the whole of $(0,+\infty)$,
and Theorem \ref{thm:existence} provides long-time  existence 
of the evolution. 

Concerning the behavior of the solutions as $t\to+\infty$, we will show that they converge,
on a suitable sequence of times, to a critical point of the energy $F_p(\Gamma)$.

\begin{thm}\label{thm:longtime}
 Assume \eqref{cond1} and 
let $\bt_{0} \in H$. Let
  $\bt= (\theta^{1}, \theta^{2}, \theta^{3})$, with $\theta^{j} \in L^{\infty}_{loc}(0,\infty; W^{1,p}(I_{j})) \cap H^{1}_{\rm loc}(0,\infty; L^{2}(I_{j}))$, and $\lambda^{1}, \lambda^{2}, \mu^{1}, \mu^{2} \in L^{\infty}(0,\infty)$ be the solutions given 
by Theorem \ref{thm:existence}. Then there exist a sequence of times $t_n\to \infty$,  Lagrange multipliers
  $\lambda_\infty^{1}, \lambda_\infty^{2}, \mu_\infty^{1}, \mu_\infty^{2} \in\R$,
and limit functions
$\bt_{\infty}=(\theta_\infty^{1}, \theta_\infty^{2}, \theta_\infty^{3})$,
with $\theta_\infty^{j} \in W^{1,p}(I_{j})$,
such that the following system holds:
 \begin{align}\nonumber
 (|\partial_{s} \theta^{1}_{\infty}|^{p-2} \partial_{s} \theta^{1}_{\infty})_{s}&= - (\lambda^{1}-\mu^{1}) \sin \theta^{1}_{\infty}  + (\lambda^{2}-\mu^{2}) \cos\theta^{1}_{\infty}\qquad \, \text{ in } I_{1}, \\\label{sislim}
 (|\partial_{s} \theta^{2}_{\infty}|^{p-2} \partial_{s} \theta^{2}_{\infty})_{s}&= \lambda^{1} \sin \theta^{2}_{\infty}  - \lambda^{2} \cos \theta^{2}_{\infty}\qquad\qquad \quad\qquad\ \ \text{ in } I_{2},\\\nonumber (|\partial_{s} \theta^{3}_{\infty}|^{p-2} \partial_{s} \theta^{3}_{\infty})_{s}&= -\mu^{1} \sin \theta^{3}_{\infty}  +\mu^{2}\cos \theta^{3}_{\infty}\qquad \qquad\qquad \qquad \text{ in } I_{3},
\end{align}
together with the boundary conditions 
\begin{equation}\label{condlim}
\partial_{s}\theta^{j}_{\infty} (0)=\partial_{s}\theta^{j}_{\infty} (L_{j}) =0\qquad \text{for $j=1,2,3$.}
\end{equation}
\end{thm}

Observe that  Theorem~\ref{thm:longtime} together with Remark~\ref{rem2.1} below yields the existence of configurations 
of planar theta-networks that are critical with respect to the elastic energy  (that is, $p=2$) 
and are subject to natural boundary conditions.
This result is relevant for the investigations undertaken in \cite{DLP18,GMP19}.
In fact we notice that, by direct method of the calculus of variations (see Section~\ref{sec:3.1} below), the energy $F_p$ always admits a global minimizers among theta-networks with curves of fixed length, moreover such a minimizer is regular (in the sense of Theorem \ref{thm:longtime}), satisfies \eqref{sislim}
and the natural boundary conditions \eqref{condlim} 
at the triple junctions.



If we do not assume \eqref{cond1} we are not able to show long-time existence, due to a technical difficulty in estimating the Lagrange multipliers $\lambda^{1}, \lambda^{2}, \mu^{1}, \mu^{2}$. However, if we assume that at least two initial curves are not flat, the same method yields the following short-time existence result.

\begin{thm}\label{thm:existence2}
Let $\bt_{0}\in H$  be such that 
\begin{equation}\label{eq:osc}
\min\left( \text{osc}_{\bar{I}_{j_1}} \, \theta_0^{j_1}(t),
\text{osc}_{\bar{I}_{j_2}} \, \theta_0^{j_2}(t)\right) \ge c >0\qquad \text{for some $j_1,j_2\in \{1,2,3\}$,}
\end{equation}
where $\text{osc}_{\bar{I_j}} \, \theta^j_0$ denotes the oscillation of $\theta^j_0$
on the interval $\bar{I_j}$. Then there exist $T=T(\bt_{0})>0$
and functions
$\bt= (\theta^{1}, \theta^{2}, \theta^{3})$, with $\theta^{j} \in L^{\infty}(0,T; W^{1,p}(I_{j})) \cap H^{1}(0,T; L^{2}(I_{j}))$,    
$\lambda^{1}, \lambda^{2}, \mu^{1}, \mu^{2} \in L^{2}(0,T)$ such that properties 
(i), (ii), (iii) of Theorem \ref{thm:existence} hold.
Moreover, letting $T_{max}$ be the maximal existence time of the evolution,  if  $T_{max}<+\infty$ there holds
\begin{equation}\label{eq:limlam}
\liminf_{t\to T_{max}^-}\, \max\left( \text{osc}_{\bar{I}_{j_1}} \, \theta^{j_1}(t),
\text{osc}_{\bar{I}_{j_2}} \, \theta^{j_2}(t)\right) = 0
\qquad \text{for some $j_1,j_2\in \{1,2,3\}$.}
\end{equation}
\end{thm}

In order to show existence of weak solutions in 
Theorems \ref{thm:existence} and \ref{thm:existence2} we apply
an implicit variational scheme to the energy $F_p(\Gamma)$ expressed in terms of the 
functions $\theta^j$. Such time-discrete schemes have been used in the study of geometric evolutions 
starting from the pivotal works by Almgren, Taylor and Wang \cite{ATW} and by 
Luckhaus and Sturzenhecker \cite{LS} in the case of the mean curvature flow.
An extension of these techniques to multiple-phase systems can be found in \cite{BK}, while an adaptation 
to the $L^2$-gradient flow for the elastic energy ($p=2$) of an open curve, which gives rise to a fourth order flow,
has been recently proposed in \cite{Fusco,Ba19}.

Starting from the work by Polden \cite{Po}, the fourth order evolution of elastic curves
has been extensively studied in the literature,
under various constraints and boundary conditions, we refer  for instance to
\cite{Lin12,DP14,NO15,NO17,DCP17,DCP14,DPS,Po,Koiso,DKS,LangerSinger1,wen95,Okabe2007,Okabe2008,Oelz11,Oelz14,Wheeler} and references therein. On the other hand, not many works treat
the second order evolution that we consider here (see \cite{Wen,Lin15,OPW18}).

The geometric evolution of planar networks is more complicated, since a network is intrinsically singular,
due to the presence of the multiple junctions, and the evolution is typically described by a system 
instead of a single equation. However,
the evolution by curvature of a network has been studied in many papers, starting from the work 
\cite{BR} where the authors first establish the short-time existence of solutions. We refer to 
\cite{MNT,MMN,MNP,MNPS16} for a discussion of the long-time existence in some particular cases,
and the formation of singularities.
An important motivation to study this flow is the analysis of models of two-dimensional multiphase systems, where the problem of the dynamics of the interfaces between different phases arises naturally. As an example, the model where the energy is simply given by the total length of the interfaces has proven useful to describe the growth of grain boundaries in polycrystalline materials (see for instance \cite{Gu,Taylor,KL} and references therein).

Regarding the fourth order evolution of elastic networks we refer to \cite{GMP18} for the short-time existence of smooth solutions
and to \cite{DLP18,GMP19} for the long-time existence, under the assumption that the tangent vectors of the three concurring curves are not collinear at a triple junction. With our approach we don't need such a condition, even if our notion of solution is considerably weaker than the one considered in  \cite{DLP18,GMP19}.

We conclude by observing that the result in Theorem \ref{thm:existence2} can be extended without significant modifications
to the case of a network of three curves with a single triple junction and three fixed endpoints, which is the situation considered in \cite{DLP18,GMP18,GMP19}: this fact is briefly discussed in Remark~\ref{rem-triod}.

The article is organized as follows: in Section~\ref{sec:prelim} we derive and motivate the system \eqref{E1}--\eqref{E4} and discuss the well-posedness of the Lagrange multipliers. In Section~\ref{secex} we investigate the construction of a weak solution via minimizing movements and provide proofs of our main results. For the reader's convenience some proofs are collected in the Appendix.


\smallskip

\noindent \textbf{Acknowledgements:} MN has been supported by GNAMPA-INdAM and by the University of Pisa Project PRA 2017-18. PP has been supported by the DFG (German Research Foundation) Projektnummer: 404870139.

\section{First variation and preliminary results}\label{sec:prelim}


Let us compute the first variation of the energy $F_{p}(\Gamma)=\sum_{i=1}^{3}F_{p}(T^{i})$.
We consider variations $T^{i}_{\epsilon}= \frac{T^{i}+\epsilon \varphi^{i}}{|T^{i} + \epsilon \varphi^{i} |}$, for $\epsilon$ small enough and $\varphi^{i} \in C^{\infty}(\bar{I}_{i},\R^{2})$, $i=1,2,3$.
Then 
$$
\frac{d}{d \epsilon} \Big|_{\epsilon =0} T^{i}_{\epsilon}=\varphi^{i} - (\varphi^{i} \cdot T^{i} )T^{i}=:\varphi^{i\perp}.
$$
Since we want to include the constraint \eqref{constraintN} we compute
$$ \frac{d}{d \epsilon} \Big|_{\epsilon =0} \left [ \left(\sum_{i=1}^{3}  F_{p}(T^{i}_{\epsilon}) \right)+ \vec{\lambda} \cdot \left(\int_{I_{1}} T^{1}_{\epsilon} ds  - \int_{I_{2}} T^{2}_{\epsilon} ds \right)  + \vec{\mu} \cdot \left(  \int_{I_{3}} T^{3}_{\epsilon} ds  -\int_{I_{1}} T^{1}_{\epsilon} ds \right) \right] =0 , $$
where $\vec{\lambda}=(\lambda^{1},\lambda^{2})$, $ \vec{\mu}=(\mu^{1}, \mu^{2}) \in \R^{2}$ are Lagrange multipliers.
A direct computation  gives
\begin{align*}
0 &= \sum_{i=1}^{3}\int_{I_{i}} | \partial_{s} T^{i}|^{p-2}\partial_{s} T^{i} \cdot \partial_{s}( \varphi^{i\perp}) ds
\\
&\quad + \vec{\lambda} \cdot \left (\int_{I_{1}}\varphi^{1\perp} ds - \int_{I_{2}}\varphi^{2\perp} ds \right) + \vec{\mu} \cdot \left(  \int_{I_{3}} \varphi^{3\perp} ds  -\int_{I_{1}}\varphi^{1\perp}  ds \right) \\
&= \int_{I_{1}} [-  \partial_{s}( | \partial_{s} T^{1}|^{p-2}\partial_{s} T^{1})  + (\vec{\lambda}-\vec{\mu}) ] \cdot \varphi^{1\perp} ds  
+ \int_{I_{2}} [-  \partial_{s}(| \partial_{s} T^{2}|^{p-2}\partial_{s} T^{2})  - \vec{\lambda} ] \cdot \varphi^{2\perp} ds 
\\
& \quad +\int_{I_{3}} [-  \partial_{s}(| \partial_{s} T^{3}|^{p-2}\partial_{s} T^{3})  + \vec{\mu} ] \cdot \varphi^{3\perp} ds 
\\
& = \int_{I_{1}}[ - \nabla_{s} (  | \partial_{s} T^{1}|^{p-2}\partial_{s} T^{1}) + ((\vec{\lambda}-\vec{\mu}) \cdot N^{1}) N^{1}] \cdot \varphi^{1} \, ds
\\
&\quad - \int_{I_{2}}[  \nabla_{s} ( | \partial_{s} T^{2}|^{p-2}\partial_{s} T^{2}) + (\vec{\lambda} \cdot N^{2}) N^{2}] \cdot \varphi^{2} \, ds\\
& \quad 
+ \int_{I_{3}}[  -\nabla_{s} (  | \partial_{s} T^{3}|^{p-2}\partial_{s} T^{3}) + (\vec{\mu} \cdot N^{3}) N^{3}] \cdot \varphi^{3} \, ds,
\end{align*}
where $\nabla_{s} \varphi= \partial_{s} \varphi - (\partial_{s} \varphi \cdot T) T$ denotes the normal component of the derivative $\partial_{s} \varphi$ and where we have used the fact that $\partial_{s}T^{i}$ vanishes at the boundary. 

This motivates the study of the \emph{second-oder} problem
\begin{align} \label{sys1}
\partial_{t} T^{1} &= \nabla_{s} ( | \partial_{s} T^{1}|^{p-2}\partial_{s} T^{1}) - ((\vec{\lambda}-\vec{\mu}) \cdot N^{1}) N^{1} \qquad \text{ in } I_{1} \times (0,t_{*}) \\  \label{sys2}
\partial_{t} T^{2} &= \nabla_{s} ( | \partial_{s} T^{2}|^{p-2}\partial_{s} T^{2}) + (\vec{\lambda} \cdot N^{2}) N^{2} \qquad \text{ in } I_{2} \times 
(0,t_{*})\\
 \label{sys3} 
\partial_{t} T^{3} &= \nabla_{s} ( | \partial_{s} T^{3}|^{p-2}\partial_{s} T^{3}) - (\vec{\mu} \cdot N^{3}) N^{3} \qquad \text{ in } I_{3} \times (0,t_{*})\\
 \label{sys4}
\partial_{s} T^{i}&=0  \quad \text{ on } \quad \partial I_{i} \times (0,t_{*}), \quad i=1,2,3,\\ 
\label{sys5}
T^{i}(\cdot, 0) &=T^{i}_{0}, \quad i=1,2,3,
\end{align}
for some $t_{*}>0$, and for smooth initial data $T^{i}_{0}$ satisfying \eqref{constraintN} and 
\begin{align}\label{bdry}
\vec{\kappa}^{i}_{0}(s)=\partial_{s}T^{i}_0(s)=0 \quad \text{ for } s\in  \{ 0, L_{i} \}, \ i=1,2,3.
\end{align}
Here $\vec{\lambda}=  \vec{\lambda}(t)=(\lambda^{1}(t), \lambda^{2}(t))$, $\vec{\mu}= \vec{\mu}(t)=(\mu^{1}(t), \mu^{2}(t))$ are  such that
\begin{align}\label{L1}
 \int_{I_{1}} |\partial_{s}T^{1}|^{p} T^{1} ds - (\vec{\lambda}-\vec{\mu}) \cdot A^{1} &= \int_{I_{2}} |\partial_{s}T^{2}|^p T^{2} ds +\vec{\lambda}\cdot A^{2}\\\label{L2}
  \int_{I_{1}} |\partial_{s}T^{1}|^{p} T^{1} ds - (\vec{\lambda}-\vec{\mu}) \cdot A^{1}&=\int_{I_{3}} |\partial_{s}T^{3}|^p T^{3} ds  -\vec{\mu}\cdot A^{3}
\end{align}
hold, where
\begin{align}
A^{i}=A^{i}(t) = \int_{I_{i}} N^{i}\otimes  N^{i} ds, \qquad i=1,2,3,
\end{align}
are  $2\times 2$ time-dependent matrices. Note that if $\det A^{i}=0$ then the Lagrange multipliers might not be well defined. We will comment on the well-posedness of the Lagrange multipliers below.
 
Under the assumption that such Lagrange multipliers exist, we observe that
as long as the flow is well defined and smooth the constraint \eqref{constraintN} is satisfied. Indeed, we have
\begin{align*}
&\frac{d}{dt} \left(\int_{I_{1}} T^{1} ds  -   \int_{I_{2}} T^{2} ds\right)= \int_{I_{1}} T_{t}^{1} ds -\int_{I_{2}} T_{t}^{2} ds \\
&=  \int_{I_{1}} \nabla_{s} (  | \partial_{s} T^{1}|^{p-2}\partial_{s} T^{1}) - ((\vec{\lambda}-\vec{\mu}) \cdot N^{1}) N^{1}  ds
-\int_{I_{2}} \nabla_{s} (  | \partial_{s} T^{2}|^{p-2}\partial_{s} T^{}) - (\vec{\lambda} \cdot N^{2}) N^{2}  ds
 \\
&= \int _{I_{1}} \partial_{s} ( | \partial_{s} T^{1}|^{p-2}\partial_{s} T^{1}) ds - \int_{I_{1}} (\partial_{s}(  | \partial_{s} T^{1}|^{p-2}\partial_{s} T^{1}) \cdot T^{1}) T^{1} ds - (\vec{\lambda} -\vec\mu )\cdot A^{1} \\
& \quad - \left( \int _{I_{2}} \partial_{s} (  | \partial_{s} T^{2}|^{p-2}\partial_{s} T^{2}) ds - \int_{I_{2}} (\partial_{s}(  | \partial_{s} T^{2}|^{p-2}\partial_{s} T^{2}) \cdot T^{2}) T^{2} ds  +\vec{\lambda}  \cdot A^{2}\right) =0,
 \end{align*}
due to  the boundary conditions \eqref{sys4} (with the convention that $|0|^{p-2}0=0$),  
the fact that $\partial_{ss}T \cdot T=- |\partial_{s}T|^{2}$, $T \cdot \partial_{s}T =0$, and \eqref{L1}. Similarly, using now~\eqref{L2}, one verifies that
\begin{align*}
\frac{d}{dt} \left(\int_{I_{3}} T^{3} ds -   \int_{I_{1}} T^{1} ds\right)=0.
\end{align*}
In other words the constraint \eqref{constraintN} is satisfied along the flow.

Note also that the energy $F_p(\Gamma)$ decreases along the flow. Indeed, using the computation above,  the fact that $T_{t}^{i}$ are normal vector fields and $\partial_{t}(\int_{I_{1}} T^{1} ds -\int_{I_{2}} T^{2} ds) =0 =\partial_{t}(\int_{I_{3}} T^{3} ds -\int_{I_{1}} T^{1} ds)$, we find
\begin{align*}
\frac{d}{dt} F_{p}(\Gamma) &= \sum_{i=1}^{3}\int_{I_{i}} [ - \nabla_{s} (  | \partial_{s} T^{i}|^{p-2}\partial_{s} T^{i})] \cdot T^{i}_{t} ds \\
&= \sum_{i=1}^{3}\int_{I_{i}} [ - \nabla_{s} (  | \partial_{s} T^{i}|^{p-2}\partial_{s} T^{i})] \cdot T^{i}_{t} ds + \vec{\lambda}  \cdot \int_{I_{1}} T^{1}_{t} ds
- \vec{\lambda}  \cdot \int_{I_{2}} T^{2}_{t} ds \\
& \quad +  \vec{\mu}  \cdot \int_{I_{3}} T^{3}_{t} ds
- \vec{\mu}  \cdot \int_{I_{1}} T^{1}_{t} ds
 \\
& = \int_{I_{1}} [ - \nabla_{s} (  | \partial_{s} T^{1}|^{p-2}\partial_{s} T^{1})] \cdot T^{1}_{t}  + ((\vec{\lambda} -\vec{\mu}) \cdot N^{1}) N^{1} \cdot T^{1}_{t} ds\\
& \quad +\int_{I_{2}} [ - \nabla_{s} (  | \partial_{s} T^{2}|^{p-2}\partial_{s} T^{2})] \cdot T^{2}_{t}  - (\vec{\lambda}  \cdot N^{2}) N^{2} \cdot T^{2}_{t} ds
\\ 
& \quad +\int_{I_{3}} [ - \nabla_{s} (  | \partial_{s} T^{3}|^{p-2}\partial_{s} T^{3})] \cdot T^{3}_{t}  + (\vec{\mu} \cdot N^{3}) N^{3} \cdot T^{3}_{t} ds\\
&= - \sum_{i=1}^{3}\int_{I_{i}} |\partial_{t} T^{i}|^{2 } ds \leq 0.
\end{align*}

The system \eqref{sys1}-- \eqref{sys5} can be converted into  a system for the scalar maps 
$\theta^{i}: I_{i} \times (0, t_{*})\to \R$ satisfying 
\begin{align}\label{defiT}
T^{i}(s,t)= (\cos \theta^{i}(s,t), \sin \theta^{i}(s,t)), \qquad  i=1,2,3.
\end{align}
Indeed, since we have
\begin{align*} 
 \nabla_{s} ( |\partial_{s} T^i|^{p-2} \partial_{s}T^i)&=  \partial_{s} ( |\partial_{s} T^i|^{p-2}) \partial_{s}T^i  +
 |\partial_{s} T^i|^{p-2} ( \partial_{s}^{2} T^i - (\partial_{s}^{2} T^i \cdot T^i) T^i)\\
 &=(|\theta^i_{s}|^{p-2})_{s} \theta^i_{s}N^{i} + |\theta^i_{s}|^{p-2} \theta^i_{ss} N^{i},
\end{align*}
%
we obtain the system 
\begin{align}\label{systheta1}
\theta^{1}_{t}&= (|\theta_{s}^{1}|^{p-2} \theta_{s}^{1})_{s} + (\lambda^{1}(t)-\mu^{1}(t)) \sin \theta^{1} - (\lambda^{2}(t)-\mu^{2}(t)) \cos \theta^{1}  \qquad \text{ in } I_{1} \times (0, t_{*})\\
\label{systheta2gen}
\theta^{2}_{t}&=  (|\theta_{s}^{2}|^{p-2} \theta_{s}^{2})_{s}  -\lambda^{1}(t) \sin \theta^{2} + \lambda^{2}(t) \cos \theta^{2}  \qquad \text{ in } I_{2} \times (0, t_{*})\\
\label{systheta3gen}
\theta^{3}_{t}&=  (|\theta_{s}^{3}|^{p-2} \theta_{s}^{3})_{s} + \mu^{1}(t) \sin \theta^{3} - \mu^{2}(t) \cos \theta^{3}  \qquad \text{ in } I_{3} \times (0, t_{*})\\ \label{systheta4gen}
 \theta_{s}^{i} (0, t) & =  \theta_{s}^{i} (L_{i}, t)=0 \qquad t\in  (0, t_{*}),\ i=1,2,3,\\
\theta^{i}( \cdot, 0) &= \theta_{0}^{i}(\cdot) \qquad i=1,2,3.
\end{align}

Regarding the Lagrange multipliers, recall that   the matrix $A^{i}$  is given by
\begin{align}\label{matA}
A^{i}=A^{i}(t)= \left(\begin{array}{cc}
\int_{I_{i}} \sin^{2} \theta^{i}  ds& -\int_{I_{i}} \sin \theta^{i}  \cos \theta^{i} ds\\
-\int_{I_{i}} \sin \theta^{i}  \cos \theta^{i} ds& \int_{I_{i}} \cos^{2} \theta^{i}  ds
\end{array} \right) =: A^{i}(\theta^{i}),
\end{align}
so that by \eqref{L1}, \eqref{L2}, we can write
\begin{align}\label{L1t}
 \int_{I_{1}} |\partial_{s}\theta^{1}|^{p} (\cos \theta^{1}, \sin \theta^{1}) ds - (\vec{\lambda}-\vec{\mu}) \cdot A^{1}(\theta^{1}) &= \int_{I_{2}} |\partial_{s}\theta^{2}|^p (\cos \theta^{2}, \sin \theta^{2})  ds +\vec{\lambda}\cdot A^{2}(\theta^{2})\\\label{L2t}
  \int_{I_{1}} |\partial_{s}\theta^{1}|^{p} (\cos \theta^{1}, \sin \theta^{1})  ds - (\vec{\lambda}-\vec{\mu}) \cdot A^{1} (\theta^{1}) &=\int_{I_{3}} |\partial_{s}\theta^{3}|^p (\cos \theta^{3}, \sin \theta^{3})  ds  -\vec{\mu}\cdot A^{3}(\theta^{3}).
\end{align}
Letting, for $i=1,2,3$,
\begin{align}\label{defG}
 G^{i}=G^{i}(\theta^{i}):=\int_{I_{i}} |\partial_{s}\theta^{i}|^{p} (\cos \theta^{i}, \sin \theta^{i}) ds,
\end{align} 
the above system reads as
\begin{align}\label{e1}
G^{1} - (\vec{\lambda}-\vec{\mu}) \cdot A^{1}&= G^{2} + \vec{\lambda}\cdot A^{2}\\ \label{e2}
G^{1} - (\vec{\lambda}-\vec{\mu}) \cdot A^{1}&= G^{3} - \vec{\mu}\cdot A^{3},
\end{align}
that is, recalling that $G^{2} + \vec{\lambda}\cdot A^{2}= G^{3} - \vec{\mu}\cdot A^{3}$,
\begin{align}\label{e1bis}
 \vec{\lambda}\cdot A^{2} +\vec{\mu}\cdot A^{3}&= G^{3} - G^{2} \\ \label{e2bis}
 -\vec{\lambda} \cdot (A^{2}+A^{1}) +\vec{\mu}\cdot A^{1}&= G^{2} - G^{1}.
\end{align}
Assuming that $A^1$, $A^2$ are invertible, we then get
\begin{align*}
\vec{\lambda}&= (G^{3}-G^{2}- \vec{\mu} \cdot A^{3}) \cdot (A^{2})^{-1}\\
\vec{\mu}&= (G^{2}-G^{1 }  +\vec{\lambda} (A^{2}+A^{1})  ) \cdot (A^{1})^{-1},
\end{align*}
which yields
\begin{align*}
\vec{\mu}\,(I + A^{3}((A^{2})^{-1}+(A^{1})^{-1}) )=(G^{2}-G^{1}) (A^{1})^{-1} +(G^{3}-G^{2}) ((A^{1})^{-1} +(A^{2})^{-1}).
\end{align*}

Observe that if  $\det (A^{i}) >0$ for $i=1,2,3$, then we can solve for $\vec{\mu}$ and $\vec{\lambda}$ and the Lagrange multipliers are well defined (simply write $(I + A^{3}((A^{2})^{-1}+(A^{1})^{-1}) )= A^{3} (\sum_{i=1}^{3} (A^{i})^{-1}) $ and use that $A^{i}$, $i=1,2,3$ are symmetric real (hence diagonalisable) and positive definite matrices (by Sylvester criterion), and that  the sum of positive definite matrices is again positive definite and hence invertible).
Note also that, by Cauchy-Schwarz inequality we have $\det A^{i} \geq 0$ $i=1,2,3$. 
A strict bound from below on the determinant is shown in \cite[Lemma~1]{Lin15} (see Lemma~\ref{lemmaLin} below), provided the considered curve is not a straight line (i.e. we need some oscillation of $\theta$).
In other words provided none of the curves is a straight line, then the system is well-posed.

The system  for the Lagrange multipliers can be solved in a slightly more general situation.
Indeed, if the matrices $A^{i}$ are such that
$\det(A^{i}) >0$ for $i=1,2$,  while $\det (A^{3}) = 0$, that is,
$\theta^{3}\equiv \theta^{*}$ for some constant $\theta^{*}$, 
we deduce that: 
\begin{itemize}
\item[(i)] $(A^{i})^{-1}$, $i=1, 2$
exist, are positive definite and symmetric; 
\item[(ii)] the matrix $M=(m_{ij})_{i,j=1,2}:= (A^{2})^{-1}+(A^{1})^{-1}$ is symmetric and positive definite, and by writing it down explicitly we infer that $m_{11}$ and $m_{22}$ are nonnegative. Moreover, since $\det M>0$, we have that
$$\sqrt{m_{11}}\sqrt{m_{22}} > |m_{21}|=|m_{12}|; $$
\item[(iii)] the symmetric matrix $A^{3}=(a_{ij})_{i,j=1,2}$ is given by $A^{3}= L_{3}\left(\begin{array}{cc}
 \sin^{2} \theta^{*}  & - \sin \theta^{*}  \cos \theta^{*} \\
- \sin \theta^{*}  \cos \theta^{*} &  \cos^{2} \theta^{*}  
\end{array} \right) 
  = L_{3}  P \left(\begin{array}{cc}
 0  &  0 \\
0 &  1  
\end{array} \right)   P^{-1}$ where $P=\left(\begin{array}{cc}
 \cos \theta^{*}  &  -\sin \theta^{*}   \\
\sin \theta^{*}   &  \cos\theta^{*}  
\end{array} \right) $ and $ P^{-1}=P^{t}$. 
In particular, note that  $a_{11}$ and $a_{22}$ are nonnegative and $\sqrt{a_{11}}\sqrt{a_{22}} \geq |a_{21}|=|a_{12}| $ holds;
\item[(iv)] $\det(A^{3}M) =\det A^{3} \det M =0$. Moreover   $(ii)$ and $(iii)$ yield
\begin{align*}
tr(A^{3}M) &=a_{11}m_{11} + a_{12} m_{21}+ a_{21}m_{12} + a_{22}m_{22} \\
& \geq a_{11}m_{11}+ a_{22}m_{22} -2 |a_{12}| |m_{12}| \\
&= (\sqrt{a_{11}}\sqrt{m_{11}} - \sqrt{a_{22} }\sqrt{m_{22}})^{2} + 2 \sqrt{a_{11}}\sqrt{m_{11}}\sqrt{a_{22} }\sqrt{m_{22}}-2 |a_{12}| |m_{12}|
\geq0. \end{align*}
This implies that the matrix $A^{3}M$ has eigenvalues $\omega_{1}=0$ and $\omega_{2}=tr(A^{3}M) \geq 0 $, and can be diagonalized.
Hence there exists an invertible  matrix $Q$ such that $$Q (A^{3}M) Q^{-1}=\left(\begin{array}{cc}
 0  & 0   \\
0  &  \omega_{2}  
\end{array} \right),  \quad \text{ where } 0 \leq \omega_{2} \leq C\left(L_{3}, L_{2},L_{1},\frac{1}{\det(A^{1})}, \frac{1}{\det(A^{2})}\right).  
$$ 
\end{itemize}
By writing 
\begin{align*}
I +A^{3}((A^{2})^{-1}+(A^{1})^{-1}) = I + A^{3}M = Q^{-1} ( I+ Q (A^{3}M) Q^{-1})Q=  Q^{-1} \left(\begin{array}{cc}
 1  & 0   \\
0  &  1+\omega_{2}  
\end{array} \right) Q 
\end{align*}
we infer that such matrix is invertible and we can solve the system for $\vec{\mu}$ and $\vec{\lambda}$. 
Moreover,  the above analysis yields that
\begin{align}\label{boundLLzero}
|\lambda| + |\mu| \leq C \left(\sum_{i=1}^{3} \int_{I_{i}} |\partial_{s} \theta^{i}|^{p} ds \right) \quad \text{ with } C=C\left(L_{1},L_{2},L_{3}, \frac{1}{\det A^{1}}, \frac{1}{\det A^{2}}\right).
\end{align}
This is a bound that is important for the analysis that follows.

Summing up  and recalling \eqref{lengths}, we have that  the Lagrange multipliers are well defined in the case
$L_{3} < \min \{ L_{1}, L_{2}\}$,
which corresponds to a situation where the curve $\gamma_{3}$ might be a straight line, whereas $\gamma_{1}$ and $\gamma_{2}$ necessarily have regions with non vanishing curvature. 
In fact, for the previous estimates on the Lagrange multipliers to hold, 
it is enough to assume that at most one of the three curves is flat.
This is the case we will mostly concentrate on, the remaining cases are briefly discussed in the following remark.

\begin{rem}\label{caso2}\rm
Let us first consider the case where $L_{3}=L_{2} <L_{1}$. As shown above the Lagrange multipliers are well defined as long as none of the curve is a straight line. If the curve $\gamma_{3}$ becomes straight, the same must happen for $\gamma_{2}$. More precisely, due to the theta-network configuration, we have that  $\gamma_{2}=\gamma_{3}$,  $A^{3}=A^{2}$  with $A^{3}$ as in case $(iii)$ discussed above, and $G^{2}=G^{3}=0$.
Summing up equations \eqref{e1} and \eqref{e2} yields
\begin{align*}
(\vec{\lambda} - \vec{\mu}) \cdot (2A^{1}+ A^{3}) &=2 G^{1}\\
(\vec{\lambda} + \vec{\mu}) \cdot A^{3} =0
\end{align*}
Since $(2A^{1}+ A^{3}) $ is positive definite and invertible, we get
\begin{align*}
(\vec{\lambda} - \vec{\mu})&= 2 G^{1} \cdot (2A^{1}+ A^{3})^{-1}\\
0&=(\vec{\lambda} + \vec{\mu}) \cdot P \left(\begin{array}{cc}
 0  & 0   \\
0  &  1  
\end{array} \right)=(\vec{\lambda} + \vec{\mu}) \cdot  \left(\begin{array}{cc}
 0  & -\sin \theta^{*}   \\
0  &  \cos \theta^{*}  
\end{array} \right)\\
& = (0,- (\lambda^{1}+\mu^{1})\sin \theta^{*}  + (\lambda^{2}+\mu^{2})\cos \theta^{*}).
\end{align*}
In this case, a solution of the system  \eqref{e1}, \eqref{e2} is given by 
\[
\vec{\lambda}=-\vec{\mu}=G^{1} \cdot (2A^{1}+ A^{3})^{-1}.
\]
However, the solution is not unique. Moreover, even if we can pick up a solution for which \eqref{boundLLzero} holds, we have no means to control the constant $C$ in \eqref{boundLLzero}  when two curves  simultaneously become flat.

\smallskip

In the case $L_{1}=L_{2}=L_{3}$, the three curves can become straight necessarily at the same time. When this happens, the energy is minimal and equal to zero, and the trivial solution of three coinciding segments is attained. In this case $G^{1}=G^{2}=G^{3}=0$, $A^{1}=A^{2}=A^{3}$ and $\vec{\lambda}=\vec{\mu}=0$ is a solution of the system \eqref{e1}, \eqref{e2}.
\end{rem}

\begin{rem}[Relation between classical formulation and $\theta$-formulation]\label{rem2.1}	\rm
For simplicity  we first consider a smooth evolution of a single curve  satisfying
$$\theta_{t}=  (|\theta_s|^{p-2}\theta_s)_{s} + \mu^{1}(t) \sin \theta - \mu^{2}(t) \cos \theta.$$
For a stationary point this implies
 \begin{align}\label{eqstat}
  0=  (|\theta_s|^{p-2}\theta_s)_{s} + \mu^{1} \sin \theta - \mu^{2}\cos \theta.\end{align}
After mutiplying by $\theta_{s}$ we obtain
 \begin{align*}
 0 =  \frac{d}{ds} \left (  \frac{p-1}{p} |\theta_{s}|^{p} -\mu^{1} \cos \theta - \mu^{2} \sin \theta\right)
 \end{align*}
 which gives
 $$ 
 \frac{p-1}{p} |\theta_{s}|^{p} - (\mu^{1} \cos \theta + \mu^{2} \sin \theta)  =: \tilde{\mu} \in \R.$$
 On the other hand, by differentiating \eqref{eqstat} we get
 \begin{align*}
 0 &= (|\theta_s|^{p-2}\theta_s)_{ss} + \theta_{s}(\mu^{1} \cos \theta + \mu^{2} \sin \theta) \\
  &=  (|\theta_s|^{p-2}\theta_s)_{ss} + \theta_{s} \left( \frac{p-1}{p} |\theta_{s}|^{p} -\tilde{\mu}\right) 
   = (|\theta_s|^{p-2}\theta_s)_{ss} +\frac{p-1}{p} |\theta_{s}|^{p}\theta_s -\tilde{\mu} \theta_{s} \\
   &=  (|\vk|^{p-2}\vk)_{ss} +\frac{p-1}{p}|\vk|^{p}\vk -\tilde{\mu} \vk.
 \end{align*}
 
Now let us consider a network satisfying the system \eqref{systheta1}, \eqref{systheta2gen},\eqref{systheta3gen}, \eqref{systheta4gen}. Reasoning as in the case of a single curve, we get that a stationary network  solves
 \begin{align*}
 (|\vk^1|^{p-2}\vk^1)_{ss} +\frac{p-1}{p}|\vk^{1}|^{p}\vk^1 -\tilde{\xi} \vk^{1}&=0, \quad \text{ in } I_{1}\\
 (|\vk^2|^{p-2}\vk^2)_{ss} +\frac{p-1}{p}|\vk^2|^{p}\vk^2 -\tilde{\lambda} \vk^{2}&=0, \quad \text{ in } I_{2}\\
 (|\vk^3|^{p-2}\vk^3)_{ss} +\frac{p-1}{p}|\vk^3|^{p}\vk^3 -\tilde{\mu} \vk^{3}&=0, \qquad \text{ in } I_{3},\\
 \vk^{i}&=0 \quad \text{ on } \partial I_{i}, \quad i=1,2,3,
 \end{align*}
 where
 \begin{align*}
 \tilde{\mu}&=\frac{p-1}{p}|\theta^{3}_{s}|^{p} - (\mu^{1} \cos \theta^{3} + \mu^{2} \sin \theta^{3}),\\
 \tilde{\lambda} &= \frac{p-1}{p}|\theta^{2}_{s}|^{p}+ (\lambda^{1} \cos \theta^{2} + \lambda^{2} \sin \theta^{2}),\\
 \tilde{\xi} &=\frac{p-1}{p}|\theta^{1}_{s}|^{p} - (( \lambda^{1}-\mu^{1}) \cos \theta^{1} + (\lambda^{2}-\mu^{2}) \sin \theta^{1}).
 \end{align*}
 Using the expressions above,  the fact that at a triple junction $\vk^{i}=\theta^{i}_{s}=0$, and the equations \eqref{systheta1}, \eqref{systheta2gen}, \eqref{systheta3gen} evaluated at a junction when the velocities $\theta^{i}_{t}=0$, one verifies that at a triple junction there holds
 \begin{align}\label{bla}
 \sum_{i=1}^{3} (|\theta^i_s|^{p-2}\theta^i_s)_{s} N^{i} = \tilde{\xi} T^{1}+ \tilde{\lambda}T^{2}+ \tilde{\mu}T^{3}.
 \end{align}
 If $p=2$, noting that $\theta^i_{ss} N^{i}= (\partial_{s} \vk^{i}) N^i = \nabla_{s} \vec{\vk}^{i}$ we see that a stationary network satisfies the natural boundary conditions at the triple junctions derived for the $L^{2}$-gradient flow of elastic networks in \cite{DLP18,GMP18,GMP19}.
\end{rem}

We now collect some important estimates which will be useful in the sequel.

\begin{lem}[\cite{Lin15}, Lemma 1]\label{lemmaLin}
Let $I=(0,L)$ and $\varphi: \bar{I} \to \R$ be a continuous function with positive oscillation $d_{0}$, i.e.
$$ \text{osc}_{\bar{I}} \, \varphi \geq d_{0}>0.$$
Suppose $\omega:[0,+\infty] \to [0,+\infty]$ is a continuous monotonic function which is a modulus of continuity of $\varphi$, i.e.,
$\omega(0)=0$ and 
$$|\varphi(s) -\varphi(\sigma)| \leq \omega(|s-\sigma|) \qquad \forall s, \sigma \in \bar{I}.
$$
Then we have the following estimates:

(i) There exists a positive constant
$$ C= \sin^{2}(\delta_{0}/4) \cdot \min \{ \omega^{-1}(\delta_{0}/4), L/2\},$$
where $\delta_{0}=\min \{ d_{0}, \pi\}$, such that
\begin{align*}
C \leq \int_{I} \sin^{2} (\varphi(s)+ \varphi_{*}) ds, \qquad C \leq \int_{I} \cos^{2} (\varphi(s)+ \varphi_{*}) ds ,
\end{align*}
for any arbitrary constant $\varphi_{*}$.

(ii) There holds
\begin{align*}
\det
 \left( \begin{array}{cc}
\int_{I} \sin^{2} \theta  ds & -\int_{I} \sin \theta  \cos \theta ds\\
-\int_{I} \sin \theta  \cos \theta ds & \int_{I} \cos^{2} \theta  ds
\end{array} \right) \geq  \frac{L}{2} C.
\end{align*}
\end{lem}
\begin{proof}
The proof given in \cite[Lemma 1]{Lin15} relies on the fact that the determinant can be written as a double integral as follows
$$\det
 \left( \begin{array}{cc}
\int_{I} \sin^{2} \theta  ds & -\int_{I} \sin \theta  \cos \theta ds\\
-\int_{I} \sin \theta  \cos \theta ds & \int_{I} \cos^{2} \theta  ds
\end{array} \right) = \frac{1}{2}\int_{I}\int_{I} \sin^{2}(\theta(s)-\theta(\sigma)) ds\, d\sigma. $$
\end{proof}

If we consider a theta-network for which at most one curve can become a line, then the angles of the remaining two curves have always  positive oscillation.
\begin{cor} \label{corca}
Suppose $ L_{3} < \min \{ L_{1}, L_{2}\}$. There there exists a constant $C>0$ such that
\begin{align*}
\det(A^{2}) \geq C, \qquad  \det(A^{1}) \geq C,
\end{align*}
where the matrices $A^{i}$, $i=1,2$, are defined as in \eqref{matA}.
The constant $C$ depends on $L_{1}, L_{2}$ and the oscillation of the angle functions $\theta^{2}$ and $\theta^{1}$.
\end{cor}

\begin{lem}\label{lem:boundLL}
Suppose $ L_{3} < \min \{ L_{1}, L_{2}\}$. Then
for the Lagrange multipliers $\vec{\lambda}$, $\vec{\mu}$ (unique solution of the system \eqref{e1}, \eqref{e2}) we have the bound
\begin{align*}
|\vec{\lambda}| + |\vec{\mu}| \leq C \left(\sum_{i=1}^{3} \int_{I_{i}} |\partial_{s} \theta^{i}|^{p} ds \right) \quad 
\end{align*}
where $C$ depends on $L_{1}, L_{2}$, $L_{3}$ and the  oscillation of the angle functions $\theta^{2}$ and $\theta^{1}$.
\end{lem}

\begin{proof}
This follows directly from \eqref{boundLLzero} and the Corollary \ref{corca}.
\end{proof}





\section{Existence of solutions}\label{secex}

From now on we shall assume that condition \eqref{cond1} holds.

\subsection{The discretization procedure}\label{sec:3.1}

\noindent\underline{\textbf{The discrete scheme}}\smallskip

Let $\bt_{0} \in H$ and $T>0$, $n \in \N$, $\tau_{n} =\frac{T}{n}$. We define a family of maps $\{ \bt_{i,n}\}_{i=0}^{n}\in H$, $\bt_{i,n}= (\theta^{1}_{i,n},\theta^{2}_{i,n},\theta^{3}_{i,n})$, inductively by making use of a minimization problem. Set $\bt_{0,n}=\bt_{0}$. For each $i \in \{ 1, \ldots, n \}$ consider the following variation problem:
\begin{align}\label{minProblem}
\min \{ E_{i,n}(\bt)\, | \, \bt \in H \}
\tag{$M_{i,n}$}
\end{align}
where
\begin{align}
E_{i,n}(\bt):= \sum_{j=1}^{3}\left(  \frac{1}{p}\int_{I_{j}} |\partial_{s}\theta^{j}|^{p} ds + \frac{1}{2\tau_{n}} \int_{I_{j}} |\theta^{j}-\theta^{j}_{i-1,n}|^{2} ds \right).
\end{align}
Existence of a minimizers $\bt \in H$ follows by standard methods in the calculus of variations taking into account that $(H, \| \cdot \|_{H})$ with $\| \bt \|_{H} :=\sum_{i=1}^{3} \| \theta^{j} \|_{W^{1,p}(I_{j})}$ is a Banach space (see \cite[Thm ~3.1]{OPW18} for similar arguments).

\bigskip

\noindent\underline{\textbf{Discrete Lagrange multipliers}}\smallskip

\noindent Let $\bt\, (= \bt_{i,n}) \in H$ denote a solution for \eqref{minProblem}. 
Moreover let $$\boldsymbol{\psi}=(\psi^{1}, \psi^{2}, \psi^{3}) \in W^{1,p}(0,L_{1}) \times W^{1,p}(0,L_{2}) \times W^{1,p}(0,L_{3}) = :\boldsymbol{W}^{1,p} $$
and define
\begin{align*}
C_{1}(\bt) &= \int_{I_{1}} \cos \theta^{1} ds - \int_{I_{2}} \cos \theta^{2} ds, \\
C_{2}(\bt) &= \int_{I_{1}} \sin \theta^{1} ds - \int_{I_{2}}  \sin \theta^{2} ds,\\
C_{3}(\bt) &= \int_{I_{3}} \cos \theta^{3} ds - \int_{I_{1}} \cos \theta^{1} ds, \\
C_{4}(\bt) &= \int_{I_{3}} \sin \theta^{3} ds - \int_{I_{1}}  \sin \theta^{1} ds.
\end{align*}

To show the existence of
 Lagrange multipliers $\vec{\lambda}_{i,n}=(\lambda^{1}_{i,n}, \lambda_{i,n}^{2})$, $\vec{\mu}_{i,n}=(\mu_{i,n}^{1}, \mu_{i,n}^{2}) \in \R^{2}$ such that
 \begin{align}\label{FV}
 \delta E_{i,n}(\bt) \boldsymbol{\psi} + \lambda^{1}_{i,n} \,\delta C_{1}(\bt)\boldsymbol{\psi} +\lambda^{2}_{i,n}\, \delta C_{2}(\bt)\boldsymbol{\psi} +\mu^{1}_{i,n}\,\delta C_{3}(\bt)\boldsymbol{\psi} +\mu^{2}_{i,n}\, \delta C_{4}(\bt)\boldsymbol{\psi}
=0 \quad \forall  \boldsymbol{\psi} \in \boldsymbol{W}^{1,p}  
 \end{align}
we consider the map
 \begin{align*}
 \R^{5 }\ni (\epsilon, \boldsymbol{t})=(\epsilon, t_{1}, t_{2}, t_{3},t_{4}) \mapsto \boldsymbol{C} (\epsilon, \boldsymbol{t})= \left(\begin{array}{c}
 C_{1}(\bt + \epsilon  \boldsymbol{\psi}  + \sum_{r=1}^{4} t_{r} \boldsymbol{\varphi}_{r})\\
 C_{2}(\bt + \epsilon  \boldsymbol{\psi}  + \sum_{r=1}^{4} t_{r} \boldsymbol{\varphi}_{r})\\
 C_{3}(\bt + \epsilon  \boldsymbol{\psi}  + \sum_{r=1}^{4} t_{r} \boldsymbol{\varphi}_{r})\\
 C_{4}(\bt + \epsilon  \boldsymbol{\psi}  + \sum_{r=1}^{4} t_{r} \boldsymbol{\varphi}_{r})
 \end{array} \right)
 \end{align*}
 for $$\boldsymbol{\varphi}_{r} =(\varphi_{r}^{1}, \varphi_{r}^{2}, \varphi_{r}^{3})\in \boldsymbol{W}^{1,p},\quad  r=1,2,3,4. $$
 Note that $\boldsymbol{C}(0,\boldsymbol{0})= \boldsymbol{0}$ since $\bt \in H$.  If the maps $\boldsymbol{\varphi}_{r} $ can be chosen such that the matrix
 \begin{align*}
 &\frac{\partial}{\partial \boldsymbol{t}} \boldsymbol{C}(0,\boldsymbol{0}) =
 \left (\begin{array}{ccc}
 \frac{\partial}{\partial t_{1}} C_{1}& \ldots &  \frac{\partial}{\partial t_{4}} C_{1}\\
 \vdots & & \vdots\\
 \frac{\partial}{\partial t_{1}} C_{4}& \ldots &  \frac{\partial}{\partial t_{4}} C_{4}
 \end{array} \right)(0,\boldsymbol{0} )
 = (\delta C_{i}(\bt) (\boldsymbol{\varphi}_{j}))_{i,j=1 \ldots 4}\\
&= \tiny{ \left(\begin{array}{cccc}
-\int_{I_{1}} \sin \theta^{1} \varphi^{1}_{1}  + \int_{I_{2}} \sin \theta^{2} \varphi^{2}_{1} & -\int_{I_{1}} \sin \theta^{1} \varphi^{1}_{2}  + \int_{I_{2}} \sin \theta^{2} \varphi^{2}_{2}& -\int_{I_{1}} \sin \theta^{1} \varphi^{1}_{3}  + \int_{I_{2}} \sin \theta^{2} \varphi^{2}_{3} & -\int_{I_{1}} \sin \theta^{1} \varphi^{1}_{4}  + \int_{I_{2}} \sin \theta^{2} \varphi^{2}_{4} \\
\int_{I_{1}} \cos \theta^{1} \varphi^{1}_{1} -\int_{I_{2}} \cos \theta^{2} \varphi^{2}_{1} & \int_{I_{1}} \cos \theta^{1} \varphi^{1}_{2} -\int_{I_{2}} \cos \theta^{2} \varphi^{2}_{2}&\int_{I_{1}} \cos \theta^{1} \varphi^{1}_{3} -\int_{I_{2}} \cos \theta^{2} \varphi^{2}_{3} & \int_{I_{1}} \cos \theta^{1} \varphi^{1}_{4} -\int_{I_{2}} \cos \theta^{2} \varphi^{2}_{4}\\
-\int_{I_{3}} \sin \theta^{3} \varphi^{3}_{1}  + \int_{I_{1}} \sin \theta^{1} \varphi^{1}_{1} & -\int_{I_{3}} \sin \theta^{3} \varphi^{3}_{2}  + \int_{I_{1}} \sin \theta^{1} \varphi^{1}_{2}& -\int_{I_{3}} \sin \theta^{3} \varphi^{3}_{3}  + \int_{I_{1}} \sin \theta^{1} \varphi^{1}_{3} & -\int_{I_{3}} \sin \theta^{3} \varphi^{3}_{4}  + \int_{I_{1}} \sin \theta^{1} \varphi^{1}_{4} \\
\int_{I_{3}} \cos \theta^{3} \varphi^{3}_{1} -\int_{I_{1}} \cos \theta^{1} \varphi^{1}_{1} & \int_{I_{3}} \cos \theta^{3} \varphi^{3}_{2} -\int_{I_{1}} \cos \theta^{1} \varphi^{1}_{2} &\int_{I_{3}} \cos \theta^{3} \varphi^{3}_{3} -\int_{I_{1}} \cos \theta^{1} \varphi^{1}_{3} & \int_{I_{3}} \cos \theta^{3} \varphi^{3}_{4} -\int_{I_{1}} \cos \theta^{1} \varphi^{1}_{4}
\end{array} \right)}
 \end{align*}
 has maximal rank, then by the implicit function theorem we have that there exist $C^{1}$-maps $\sigma_{r}$, $r=1,2,3,4$, defined in a neighborhood of zero, such that $\sigma_{r}(0)=0$ for $r=1,2,3,4,$ and 
 $$\boldsymbol{C}(\epsilon, \sigma_{1}(\epsilon), \ldots, \sigma_{4}(\epsilon))= \boldsymbol{0}  \qquad \text{ for } \epsilon \in (-\epsilon_{0}, \epsilon_{0}),
 $$
 i.e., $\bt + \epsilon  \boldsymbol{\psi}  + \sum_{r=1}^{4} \sigma_{r}(\epsilon) \boldsymbol{\varphi}_{r} \in H$ for $\epsilon \in (-\epsilon_{0}, \epsilon_{0})$.
 Differentiation in $\epsilon$ of the above equation gives
 \begin{align*}
 \left( \begin{array}{c}
 \sigma_{1}'(0)\\
 \vdots\\
 \sigma_{4}' (0)
 \end{array} \right) = - \left(\frac{\partial}{\partial \boldsymbol{t}} \boldsymbol{C}(0,\boldsymbol{0})\right)^{-1}
 \left( \begin{array}{c}
 \delta  C_{1} (\bt) \boldsymbol{\psi}\\
 \vdots\\
 \delta  C_{4} (\bt) \boldsymbol{\psi}
 \end{array} \right)
 \end{align*}
 so that, from the minimality of $\bt$ we infer
 \begin{align*}
 0 &=\frac{d}{d\epsilon} \Big|_{\epsilon=0} E_{i,n} \left(\bt + \epsilon  \boldsymbol{\psi}  + \sum_{r=1}^{4} \sigma_{r}(\epsilon) \boldsymbol{\varphi}_{r} \right) 
 = \delta E_{i,n} (\bt) \boldsymbol{\psi}+ \sum_{r=1}^{4} \sigma_{r}'(0) \delta E_{i,n} (\bt) \boldsymbol{\varphi}_{r}\\
 & =\delta E_{i,n} (\bt) \boldsymbol{\psi} - \sum_{l=1}^{4} \left ( \sum_{r=1}^{4} \left(\frac{\partial}{\partial \boldsymbol{t}} \boldsymbol{C}(0,\boldsymbol{0})\right)^{-1}_{rl}  \delta E_{i,n} (\bt) \boldsymbol{\varphi}_{r} \right)\, \delta  C_{l} (\bt) \boldsymbol{\psi}.
 \end{align*}
 It follows that \eqref{FV} holds with
 \begin{align*}
 \lambda_{i,n}^{1} &= -\sum_{r=1}^{4} \left(\frac{\partial}{\partial \boldsymbol{t}} \boldsymbol{C}(0,\boldsymbol{0})\right)^{-1}_{r1}  \delta E_{i,n} (\bt) \boldsymbol{\varphi}_{r} \\
 \lambda_{i,n}^{2} &= -\sum_{r=1}^{4} \left(\frac{\partial}{\partial \boldsymbol{t}} \boldsymbol{C}(0,\boldsymbol{0})\right)^{-1}_{r2}  \delta E_{i,n} (\bt) \boldsymbol{\varphi}_{r}\\
 \mu_{i,n}^{1} &= -\sum_{r=1}^{4} \left(\frac{\partial}{\partial \boldsymbol{t}} \boldsymbol{C}(0,\boldsymbol{0})\right)^{-1}_{r3}  \delta E_{i,n} (\bt) \boldsymbol{\varphi}_{r}\\
 \mu_{i,n}^{2} &= -\sum_{r=1}^{4} \left(\frac{\partial}{\partial \boldsymbol{t}} \boldsymbol{C}(0,\boldsymbol{0})\right)^{-1}_{r4}  \delta E_{i,n} (\bt) \boldsymbol{\varphi}_{r}
 \end{align*}
 or equivalently
 \begin{align*}
 (\lambda_{i,n}^{1},\lambda_{i,n}^{2}, \mu_{i,n}^{1}, \mu_{i,n}^{2}) \left(\frac{\partial}{\partial \boldsymbol{t}} \boldsymbol{C}(0,\boldsymbol{0})\right) = -(\delta E_{i,n} (\bt) \boldsymbol{\varphi}_{1},\delta E_{i,n} (\bt) \boldsymbol{\varphi}_{2},\delta E_{i,n} (\bt) \boldsymbol{\varphi}_{3},\delta E_{i,n} (\bt) \boldsymbol{\varphi}_{4}).
 \end{align*}
 By letting
 \begin{align*}
 \boldsymbol{\varphi}_{1} &:= (0, \sin \theta^{2}, -\sin \theta^{3}),  \qquad \boldsymbol{\varphi}_{3} := ( \sin \theta^{1},- \sin \theta^{2},0)\\
 \boldsymbol{\varphi}_{2} &:= (0, -\cos \theta^{2}, \cos \theta^{3}), \qquad \boldsymbol{\varphi}_{4} := (- \cos \theta^{1}, \cos \theta^{2},0)
 \end{align*}
 we obtain that 
 \begin{align*}
 \left(\frac{\partial}{\partial \boldsymbol{t}} \boldsymbol{C}(0,\boldsymbol{0})\right)  = \left( \begin{array}{c|c}
 A^{2} & -(A^{1}+ A^{2})\\
 \hline
 A^{3} & A^{1}
 \end{array} \right)
 \end{align*}
 with $A^{i} \in \R^{2 \times 2}$ as in \eqref{matA}. Moreover, we compute
 \begin{align*}
 (\delta E_{i,n}& (\bt) \boldsymbol{\varphi}_{1},\delta E_{i,n} (\bt) \boldsymbol{\varphi}_{2}) =G^{2} -G^{3} \\
 & \quad +\frac{1}{\tau_{n}} \int_{I_{2}} (\theta^{2}-\theta^{2}_{i-1,n}) (\sin \theta^{2}, - \cos \theta^{2})ds -\frac{1}{\tau_{n}} \int_{I_{3}} (\theta^{3}-\theta^{3}_{i-1,n}) (\sin \theta^{3}, - \cos \theta^{3}) ds,\\
 (\delta E_{i,n}& (\bt) \boldsymbol{\varphi}_{3},\delta E_{i,n} (\bt) \boldsymbol{\varphi}_{4}) =G^{1} -G^{2} \\
 & \quad -\frac{1}{\tau_{n}} \int_{I_{2}} (\theta^{2}-\theta^{2}_{i-1,n}) (\sin \theta^{2}, - \cos \theta^{2})ds +\frac{1}{\tau_{n}} \int_{I_{1}} (\theta^{1}-\theta^{1}_{i-1,n}) (\sin \theta^{1}, - \cos \theta^{1}) ds
 \end{align*}
 where $G^{i}$ is as in \eqref{defG}. Therefore the  Lagrange multipliers solve
 \begin{align}\label{eqLdis1}
 (\lambda_{i,n}^{1}, \lambda_{i,n}^{2}) \cdot A^{2} + (\mu_{i,n}^{1}, \mu_{i,n}^{2})\cdot A^{3} &=G^{3} -G^{2} + R_{i,n}^{23}\\ \label{eqLdis2}
  -(\lambda_{i,n}^{1}, \lambda_{i,n}^{2}) \cdot (A^{1}+A^{2}) + (\mu_{i,n}^{1}, \mu_{i,n}^{2})\cdot A^{1} &=G^{2} -G^{1} + R_{i,n}^{21},
 \end{align}
 where we  set
 \begin{align*}
 - R_{i,n}^{23}&:=\frac{1}{\tau_{n}} \int_{I_{2}} (\theta^{2}-\theta^{2}_{i-1,n}) (\sin \theta^{2}, - \cos \theta^{2})ds -\frac{1}{\tau_{n}} \int_{I_{3}} (\theta^{3}-\theta^{3}_{i-1,n}) (\sin \theta^{3}, - \cos \theta^{3}) ds ,\\
- R_{i,n}^{21}&:=-\frac{1}{\tau_{n}} \int_{I_{2}} (\theta^{2}-\theta^{2}_{i-1,n}) (\sin \theta^{2}, - \cos \theta^{2})ds +\frac{1}{\tau_{n}} \int_{I_{1}} (\theta^{1}-\theta^{1}_{i-1,n}) (\sin \theta^{1}, - \cos \theta^{1}) ds .
 \end{align*}
 Recalling the system \eqref{e1bis}, \eqref{e2bis} and the subsequent discussion concerning its solvability, we can conclude that (under assumption \eqref{cond1}) the above system is solvable, that is, the matrix $\left(\frac{\partial}{\partial \boldsymbol{t}} \boldsymbol{C}(0,\boldsymbol{0})\right)$ has maximal rank.
Moreover, similarly to  Lemma~\ref{lem:boundLL} we infer that
\begin{align}\label{boundLLdis}
|\vec{\lambda}_{i,n}| + |\vec{\mu}_{i,n}| \leq 
C \left(\sum_{j=1}^{3} \int_{I_{j}} |\partial_{s} \theta^{j}|^{p}\right) + \frac{C}{\tau_{n}}
\sum_{j=1}^{3} \int_{I_{j}}  |\theta^{j}-\theta^{j}_{i-1,n}| ds,
\end{align} 
 where $C$ has the same dependencies given in  Lemma~\ref{lem:boundLL}, and $\bt=(\theta^{1}, \theta^{2}, \theta^{3})=\bt_{i,n}$ is solution to \eqref{minProblem}.
 
 \bigskip

\noindent\underline{\textbf{Regularity}}\smallskip 

\noindent Let $\bt_{i,n} \in H$ be a solution to \eqref{minProblem}. Since \eqref{FV} and \eqref{boundLLdis} hold for $\bt=\bt_{i,n}$   it follows  that  the map $|\partial_{s} \theta^{j}|^{p-2} \partial_{s} \theta^{j} \in L^{\frac{p}{p-1}}(L_{j})$ admits weak derivative in $L^{1}(L_{j})$ with 
 \begin{align}\label{regularity}
 |(|\partial_{s} \theta^{j}|^{p-2} \partial_{s} \theta^{j})_{s} | \leq  C \left|\frac{\theta^{j} - \theta^{j}_{i-1,n}}{ \tau_{n}}\right| + |\lambda^{1}_{i,n}| + |\lambda^{2}_{i,n}| + |\mu^{1}_{i,n}| + |\mu^{1}_{i,n}|.   
 \end{align}
Moreover, the natural boundary conditions 
\begin{align}\label{nbc99}
\partial_{s} \theta^{j}(s)=0 \quad \text{ for } s \in \{0, L_{j}\}
\end{align}  
hold for $j=1,2,3$.
 

\bigskip

\noindent\underline{\textbf{Definition of approximating functions}}\smallskip

\noindent
First of all let us introduce some notation. We denote by $\boldsymbol{V}_{i,n}=(V_{i,n}^{1}, V_{i,n}^{2}, V_{i,n}^{3})$ the discrete velocity
\begin{align*}
\boldsymbol{V}_{i,n}:= \frac{\bt_{i,n} -\bt_{i-1,n}}{\tau_{n}}
\end{align*}
We will need maps that interpolate the three components of   our maps $\{ \bt_{i,n}\}_{i=0,\ldots, n}$ linearly in time:
\begin{dfn}\label{definition:2.4}
Let $\bt_{n}: I_{1} \times  I_{2} \times I_{3} \times [0,T] \to \R^{3}$ be defined by
\begin{align*}
\bt_{n}(\boldsymbol{s},t) :=\bt_{i,n-1}(\boldsymbol{s}) + (t-(i-1)\tau_{n})\boldsymbol{V}_{i,n} (\boldsymbol{s}) 
\end{align*}
if $(\boldsymbol{s},t)=(s_{1},s_{2},s_{3}, t) \in I_{1} \times  I_{2} \times I_{3} \times [(i-1)\tau_{n}, \tau_{n}] $
for $i=1, \ldots,n$.
\end{dfn}
We will need also piecewise constant interpolations, that is,
\begin{dfn}\label{definition:2.5}
Let $\bar{\bt}_{n}, \underline{\bt}_{n},\boldsymbol{V}_{n} : I_{1} \times  I_{2} \times I_{3} \times [0,T] \to \R^{3}$ be defined by
\begin{align*}
\underline{\bt}_{n}(\boldsymbol{s},t) &:=\bt_{i-1,n}(\boldsymbol{s}) ,\\
\bar{\bt}_{n} (\boldsymbol{s},t) &:=\bt_{i,n}(\boldsymbol{s}) ,\\
\boldsymbol{V}_{n}(\boldsymbol{s},t) &:= \boldsymbol{V}_{i,n}(\boldsymbol{s})
\end{align*}
if $(\boldsymbol{s},t)=(s_{1},s_{2},s_{3}, t) \in I_{1} \times  I_{2} \times I_{3} \times [(i-1)\tau_{n}, \tau_{n}] $
for $i=1, \ldots,n$.
\end{dfn}
Similarly for the discrete Lagrange multipliers (recall \eqref{FV}) we define
\begin{dfn}
Let $\vec{\boldsymbol{\lambda}}_{n}, \vec{\boldsymbol{\mu}}_{n}:[0,T] \to \R^{2}$ be defined by
\begin{align*}
\vec{\boldsymbol{\lambda}}_{n} (t)=(\lambda_{n}^{1}(t), \lambda_{n}^{2}(t)) &:= \vec{\lambda}_{i,n},\\
\vec{\boldsymbol{\mu}}_{n} (t)=(\mu_{n}^{1}(t), \mu_{n}^{2}(t)) &:= \vec{\mu}_{i,n}
\end{align*}
if $t \in[(i-1)\tau_{n}, \tau_{n}] $
for $i=1, \ldots,n$.
\end{dfn}

To keep the notation as simple as possible we adopt from now on following conventions.
For $\bt=(\theta^{1}, \theta^{2}, \theta^{3})$ in an appropriate function space and $q \in [1, \infty)$ we write
\begin{align*}
\int_{I} | \bt|^{q} ds &:= \sum_{j=1}^{3} \int_{I_{j}} | \theta^{j}|^{q} ds, \qquad \int_{I} |\partial_{s} \bt|^{q} ds := \sum_{j=1}^{3} \int_{I_{j}} |\partial_{s} \theta^{j}|^{q} ds,\\
\int_{0}^{T}\int_{I} |\bt|^{q} ds dt &:=\sum_{j=1}^{3} \int_{0}^{T}\int_{I_{j}} | \theta^{j}|^{q} ds dt.
\end{align*}

\bigskip

\noindent\underline{\textbf{Uniform bounds for the approximating functions}}\smallskip

\noindent
We now derive some uniform bounds for solutions of \eqref{minProblem}.

\begin{thm}\label{thm:3.2}
Assume \eqref{cond1}. Let $\bt_{0} \in H$ and $T >0$ be given. Let $\bt_{i,n} \in H$ be the solution of \eqref{minProblem} and let $\vec{\lambda}_{i,n}$, $\vec{\mu}_{i,n} \in \R^{2}$ be the Lagrange multipliers fulfilling \eqref{FV}. Upon recalling the definitions and convention given above,
write $$ D(\bt_{i,n}):= \frac{1}{p}\int_{I} |\partial_{s} \bt_{i,n}|^{p} ds.$$
Then
we have that
\begin{align*}
D(\bt_{i,n} ) \leq D(\bt_{i-1,n}) &\leq D(\bt_{0}) \qquad \text{ for all } i=1, \ldots,n,\\
\frac{1}{2}\int_{0}^{T} \int_{I} |\boldsymbol{V_{n}}|^{2} ds \,dt &\leq D(\bt_{0}),\\
\int_{0}^{T} |\vec{\boldsymbol{\lambda}}_{n}|^{2} (t) dt+ \int_{0}^{T} |\vec{\boldsymbol{\mu}}_{n} |^{2}(t)  dt&\leq C \big(T  D(\bt_{0})
+ 1\big)D(\bt_{0}),\\
\int_{I} |\bt_{i,n}|^{2} ds &\leq C \int_{I} |\bt_{0}|^{2} ds + C  T D(\bt_{0}),
\end{align*}
where $C$ has the same dependencies of the constant appearing in Lemma~\ref{lem:boundLL}.
\end{thm}

\begin{proof}
We let
$$ P_{i,n} (\bt):=\sum_{j=1}^{3}\left(  \frac{1}{2\tau_{n}} \int_{I_{j}} |\theta^{j}-\theta^{j}_{i-1,n}|^{2} ds \right)$$
so that $E_{i,n}(\bt)= D(\bt) + P_{i,n} (\bt)$.
The proof of the first statement follows by an induction argument. Fix $i \in \{ 1, \ldots, n \}$ and assume that $D(\bt_{j,n}) \leq D(\bt_{0})$ for all $j=1, \ldots, i-1$.
Then it follows from the minimality of $\bt_{i,n}$ that
\begin{align*}
D(\bt_{i,n}) \leq D(\bt_{i,n}) + P_{i,n} (\bt_{i,n}) = E_{i,n}(\bt_{i,n}) \leq E_{i,n}(\bt_{i-1,n}) = D(\bt_{i-1,n}).
\end{align*}
This gives the first statement.
Next observe that from
\begin{align}\label{uff}
 P_{i,n} (\bt_{i,n}) \leq  D(\bt_{i-1,n}) -  D(\bt_{i,n}) \end{align}
we obtain
\begin{align*}
\int_{0}^{T} \int_{I} |\boldsymbol{V_{n}}|^{2} ds dt &= \sum_{i=1}^{n} \tau_{n} \int_{I} |\boldsymbol{V_{i,n}}|^{2} ds =
2  \sum_{i=1}^{n} P_{i,n}(\bt_{i,n}) \\
&\leq 2 \sum_{i=1}^{n} (D(\bt_{i-1,n}) -  D(\bt_{i,n})) \leq  2 D(\bt_{0})
\end{align*}
and the second statement follows.
From \eqref{boundLLdis} we infer that
\begin{align}\label{dazwischen}
|\vec{\lambda}_{i,n}| + |\vec{\mu}_{i,n}| \leq C D(\bt_{i,n}) + C \int_{I} |\boldsymbol{V}_{i,n}| ds \leq C D(\bt_{0}) + C \left(\int_{I} |\boldsymbol{V}_{i,n}|^{2} ds\right)^{1/2},
\end{align}
which gives the third statement after
squaring and integrating in time. 
Finally, observe that for $j=1,2,3$ we can write
\begin{align*}
\|\theta^{j}_{i,n} \|_{L^{2}(I_{j})} &\leq  \|\theta^{j}_{0,n} \|_{L^{2}(I_{j})} +\sum_{r=1}^{i} \| \theta^{j}_{r,n} - \theta^{j}_{r-1,n} \|_{L^{2}(I_{j})}   =  \|\theta^{j}_{0,n} \|_{L^{2}(I_{j})} +\sum_{r=1}^{i} \sqrt{\tau_{n}} \left\| \frac{\theta^{j}_{r,n} - \theta^{j}_{r-1,n}}{\sqrt{\tau_{n}}} \right \|_{L^{2}(I_{j})} \\
& \leq  \|\theta^{j}_{0,n} \|_{L^{2}(I_{j})}  + \sqrt{i \tau_{n}} \left ( \sum_{r=1}^{i} \int_{I_{j}} \frac{|\theta^{j}_{r,n} - \theta^{j}_{r-1,n}|^{2}}{\tau_{n}}  ds \right)^{1/2} \\
& \leq \|\theta^{j}_{0,n} \|_{L^{2}(I_{j})}  + \sqrt{2T}\left( \sum_{r=1}^{i}   P_{r,n}(\bt_{r,n})\right)^{1/2} \leq  \|\theta^{j}_{0,n} \|_{L^{2}(I_{j})}  + \sqrt{2T D(\bt_{0})},
\end{align*}
where we have used again \eqref{uff}. The last statement follows.
\end{proof}

\subsection{Convergence of the scheme}

Having achieved some uniform bounds for the approximating maps, it is possible to pass to the limit as $n\to\infty$. 
The following three Lemmas are similar to the ones obtained in \cite[Lemma~3.11, Lemma~3.12, Lemma~3.13]{OPW18}.  
For the reader's convenience, we include the proofs in the Appendix.
We point out that condition \eqref{cond1} is not needed to prove these results, since the Lagrange multipliers are not involved.

\begin{lem} \label{lem:4.1} 
Let $\bt_{0} \in H$ and $T>0$ be as in Theorem~\ref{thm:3.2}.
Let $\bt_{n}=(\theta_{n}^{1}, \theta_{n}^{2}, \theta_{n}^{3})$ be the piecewise linear interpolation of $\{\bt_{i,n}\}$ given in Definition~\ref{definition:2.4}. 
Then, for $j=1,2,3$, there exists a map 
\begin{align*}
\theta^{j} \in L^{\infty}(0,T; W^{1,p}(I_{j})) \cap H^{1}(0,T; L^{2}(I_{j})) 
\end{align*}
such that  
\begin{align} \label{eq:4.1}
\frac{1}{2}\int^{T}_{0} \int_{I_{j}} | \pd_t \theta^{j}(s,t) |^{2} \, ds dt \leq D(\bt_{0}),
\end{align}

\begin{align}\label{eq:4.1bis}  
\sup_{(0,T)} \|\partial_{s} \theta^{j} \|_{L^{p}(I_{j})} &\leq C=C(D(\bt_{0}), p),\\
\sup_{(0,T)} \| \theta^{j} \|_{W^{1,p}(I_{j})} &\leq C=C(p, L_{j},T, D(\bt_{0}), \| \theta^{j}_{0}\|_{L^{2}(I_{j})}),
\end{align}
and, for a subsequence which we still denote by $\theta^{j}_{n}$,
\begin{align} \label{eq:4.2}
\begin{cases}
&\theta^{j}_{n} \rightharpoonup \theta^{j} \quad \text{weakly\,$\ast$ in} \quad L^{\infty}(0,T; W^{1,p}(I_{j})), \\
&\theta^{j}_{n} \rightharpoonup \theta^{j} \quad \text{weakly in} \quad  H^{1}(0,T; L^{2}(I_{j})), 
\end{cases}
\quad \text{as} \quad n \to \infty. 
\end{align}
Moreover, for $\alpha=\min \{\frac{1}{4}, \frac{p-1}{2p} \} $ we have that 
\begin{align} \label{eq:4.3}
\theta^{j}_{n} \to \theta^{j} \quad \text{in} \quad  C^{0,\alpha}([0,T] \times I_{j}). 
\end{align}
In particular, $\theta^{j}(\cdot,t) \to \theta^{j}_{0}(\cdot)$ in $C^{0}$ as $t \downarrow 0$. 
\end{lem}

A direct consequence of equation \eqref{eq:4.3} of Lemma \ref{lem:4.1} and  Lemma~\ref{corPaola2} is the following 
\begin{cor}\label{corpaola}
Under the assumptions of Lemma \ref{lem:4.1},
for all $j=1,2,3$ and $n\in \mathbb N$ there holds
\[
\text{osc}_{\bar{I_j}} \, \theta^{j}_{n},\ \text{osc}_{\bar{I_j}} \, \theta^{j}  \in C^\alpha([0,T]).
\]
\end{cor}
In particular, using  Definition~\ref{definition:2.4},  we can assert that the oscillations of the maps $\theta^{j}_{i,n}$ are close to the oscillation of $\theta^{j}_{0,n}$ if $T$ is chosen sufficiently small. This fact will be used in the proof of Theorem~\ref{thm:existence2}.

\begin{lem} \label{lem:4.2}
Let $\bt_{0} \in H$ and $T>0$ be as in Theorem~\ref{thm:3.2}.
Let $\bar{\bt}_{n}=(\bar{\theta}^{1},\bar{\theta}^{2},\bar{\theta}^{3} )$, $\underline{\bt}_{n}=(\underline{\theta}^{1},\underline{\theta}^{2},\underline{\theta}^{3} )$ be the piecewise constant interpolations of $\{ \bt_{i,n} \}$ as given in Definition~\ref{definition:2.5}. 
Then we have 
\begin{align} \label{eq:4.13}
\bar{\theta}_{n}^{j} \to \theta^{j}  \quad \text{ and } \quad  \underline{\theta}^{j}_{n} \to \theta^{j} \text{ in } C^{0}([0,T] \times I_{j}), \quad j=1,2,3, 
\end{align} 
where $\theta^{j}$, $j=1,2,3$, denote the maps obtained in Lemma \ref{lem:4.1}.   
Moreover, it holds that 
\begin{align} \label{eq:4.14}
 \partial_{s}\bar{\theta}^{j}_{n} \rightharpoonup  \partial_{s}\theta^{j} \quad  \text{ and } \quad  \partial_{s}\underline{\theta}_{n}^{j} \rightharpoonup  \partial_{s}\theta^{j}  \quad  \text{weakly in} \quad L^p(0,T; L^{p}(I_{j})) \quad \text{as} \quad n \to \infty. 
\end{align}
\end{lem}

\begin{lem} \label{lem:4.4}
Let $\bar{\bt}_{n}=(\bar{\theta}_{n}^{1}, \bar{\theta}_{n}^{2}, \bar{\theta}_{n}^{3})$ be the piecewise constant interpolation of $\{\bt_{i,n}\}$ given in Definition~\ref{definition:2.5}. 
and let the assumptions of Lemma~\ref{lem:4.1} hold. 
Then,  for $j=1,2,3$,
it holds that 
\begin{align*}
\int^{T}_{0} \int_{I_{j}} | (\bar{\theta}^{j}_{n})_{s}|^{p-2}  (\bar{\theta}^{j}_{n})_{s}  \cdot  \vp_{s} \, ds dt 
 \to \int^{T}_{0} \int_{I_{j}} | \theta^{j}_{s} |^{p-2}  \theta^{j}_{s}  \cdot  \vp_{s} \, ds dt 
 \quad \text{as} \quad n \to \infty 
\end{align*}
for any $\vp \in L^{\infty}(0,T; W^{1,p}(I_{j}))$. 
\end{lem}

We can now prove our main existence result.

\begin{proof}[Proof of Theorem~\ref{thm:existence}] 
(i) Equation \eqref{FV}  yields that for any $\boldsymbol{\varphi}=(\vp^{1},\vp^{2}, \vp^{3})$ with $\varphi^{j} \in L^{\infty} (0,T; W^{1,p}(I_{j}))$, $j=1,2,3,$  and (for almost every) $t \in  ((i-1)\tau_{n},  i \tau_{n}]$, $i=1, \ldots, n$ we have
\begin{align*}
0&=\sum_{j=1}^{3}\int_{I_{j}}V^{j}_{n}(s,t) \varphi^{j}(s,t) ds 
+ \sum_{j=1}^{3}\int_{I_{j}} | (\bar{\theta}^{j}_{n})_{s}|^{p-2}  (\bar{\theta}^{j}_{n})_{s}  \, ( \varphi^{j})_{s} \,ds \\
& \qquad-({\lambda}^{1}_{n}(t) - \mu^{1}_{n}(t))\int_{I_{1}} \sin (\bar{\theta}_{n}^{1}) \, \varphi^{1} ds 
+ (\lambda^{2}_{n}(t) -\mu_{n}^{2}(t))\int_{I_{1}} \cos (\bar{\theta}_{n}^{1} )\,  \varphi^{1} ds \\
& \qquad +\lambda^{1}_{n}(t) \int_{I_{2}} \sin (\bar{\theta}_{n}^{2}) \, \vp^{2} ds  -\lambda^{2}_{n}(t) \int_{I_{2}} \cos (\bar{\theta}_{n}^{2}) \, \vp^{2} ds \\
& \qquad -\mu^{1}_{n}(t) \int_{I_{3}} \sin (\bar{\theta}_{n}^{3}) \, \vp^{3} ds  +\mu^{2}_{n}(t) \int_{I_{3}} \cos (\bar{\theta}_{n}^{3}) \, \vp^{3} ds 
\end{align*}
so that integration in time yields
\begin{align*}
0&=\sum_{j=1}^{3} \int_{0}^{T}\int_{I_{j}}V^{j}_{n}(s,t) \varphi^{j}(s,t) ds   dt
+ \sum_{j=1}^{3} \int_{0}^{T}\int_{I_{j}} | (\bar{\theta}^{j}_{n})_{s}|^{p-2}  (\bar{\theta}^{j}_{n})_{s}  \, ( \varphi^{j})_{s} \,ds dt \\
& \qquad-\int_{0}^{T}({\lambda}^{1}_{n}(t) - \mu^{1}_{n}(t))\int_{I_{1}} \sin (\bar{\theta}_{n}^{1}) \, \varphi^{1} ds  dt
+ \int_{0}^{T}(\lambda^{2}_{n}(t) -\mu_{n}^{2}(t))\int_{I_{1}} \cos (\bar{\theta}_{n}^{1} )\,  \varphi^{1} ds dt \\
& \qquad +\int_{0}^{T}\lambda^{1}_{n}(t) \int_{I_{2}} \sin (\bar{\theta}_{n}^{2}) \, \vp^{2} ds  dt -\int_{0}^{T}\lambda^{2}_{n}(t) \int_{I_{2}} \cos (\bar{\theta}_{n}^{2}) \, \vp^{2} ds dt \\
& \qquad -\int_{0}^{T}\mu^{1}_{n}(t) \int_{I_{3}} \sin (\bar{\theta}_{n}^{3}) \, \vp^{3} ds dt  +\int_{0}^{T}\mu^{2}_{n}(t) \int_{I_{3}} \cos (\bar{\theta}_{n}^{3}) \, \vp^{3} ds dt 
\end{align*}
for any $\boldsymbol{\varphi} \in L^{\infty} (0,T; \boldsymbol{W}^{1,2})$.
We now let $n \to \infty$. The first two integrals are dealt with in Lemma~\ref{lem:4.1} and  Lemma~\ref{lem:4.4}.
By the uniform bound given Theorem~\ref{thm:3.2} we have that there exist $\lambda^{1}, \lambda^{2}, 
\mu^{1}, \mu^{2} \in L^{2}(0,T)$ such that 
\begin{align}\label{convLM}
\lambda^{j}_{n} \rightharpoonup \lambda^{j} \text{ weakly in } L^{2}(0,T), \qquad \mu^{j}_{n} \rightharpoonup \mu^{j} \text{ weakly in } L^{2}(0,T)
\end{align}
for $j=1,2$.
Since $v_{n}(t):=\int_{I_{1}} \sin (\bar{\theta}^{1}_{n}) \varphi^{1}(s,t) ds \to \int_{I_{1}} \sin (\theta^{1}) \varphi^{1}(s,t) ds =:v(t) $ by  Lemma~\ref{lem:4.2}, and $|v_{n}| \leq C(\vp^{1})$, then also $v_{n} \to v$ in $L^{2}(0,T)$ and 
we infer that
$$\int_{0}^{T}\lambda^{1}_{n}(t) \int_{I_{1}} \sin (\bar{\theta}_{n}^{1})\, \varphi^{1} ds dt   \to \int_{0}^{T} \lambda^{1}(t)\int_{I_{1}} \sin (\theta^{1}) \,\varphi^{1} ds dt $$
for $n \to \infty$. The other integrals with the Lagrange multipliers are  treated in a similar way and the first statement follows. 

(ii) Equations \eqref{E1}, \eqref{E2}, \eqref{E3}, and the natural boundary conditions \eqref{E4} follow directly from \eqref{eqtest} by choosing test functions of the form $\varphi^{j}(s,t)=\tilde{\varphi}(t)\psi^{j}(s)$ with $\psi^{j } \in W^{1,p}(I_{j})$ and $ \tilde{\varphi}\in C^{\infty}_{0}(0,T)$. Also we exploit the fact that given any map $f \in L^{1}(I)$ with $f_{s} \in L^{2}(I)$ and  $I \subset \R$ bounded interval, it follows from embedding theory that $f \in H^{1}(I)$.

(iii) By construction we have that ${\bt}_{i,n} \in H$, so that  
\begin{align*}
\int_{I_{1}} (\cos \bar{\theta}^{1}_{n}, \sin \bar{\theta}^{1}_{n})  ds=\int_{I_{2}} (\cos \bar{\theta}^{2}_{n}, \sin \bar{\theta}^{2}_{n}) ds = \int_{I_{3}} (\cos \bar{\theta}^{3}_{n}, \sin \bar{\theta}^{3}_{n})  ds
\end{align*}
for  all $t \in  ((i-1)\tau_{n},  i \tau_{n}]$, $i=1, \ldots, n$.
Passing to the limit as $n\to\infty$ and using \eqref{eq:4.13} we obtain  \eqref{constraintlimit}.
\end{proof}

We now show that the Lagrange multipliers in Theorem \ref{thm:existence} are uniformly bounded in time.

\begin{prop}
Let $\bt_{0} \in H$, $T>0$, $\bt=(\theta^{1}, \theta^{2}, \theta^{3})$, $\vec{\lambda}$ and $\vec{\mu}$ be as in Theorem~\ref{thm:existence}. Then we have that the system \eqref{e1bis}, \eqref{e2bis} holds for almost every time and
\begin{align}\label{eq:3.20}
\| \vec{\lambda}\|_{L^{\infty}(0,T)} + \| \vec{\mu}\|_{L^{\infty}(0,T)} \leq C (D(\bt_{0}), p).
\end{align}
The constant C has the same dependencies as in Lemma~\ref{lem:boundLL}.
\end{prop}

\begin{proof}
Testing the weak formulation \eqref{eqtest} with $\boldsymbol{\varphi}(\boldsymbol{s},t)=(-\tilde{\varphi} \sin \theta^{1} , 0,0 )$ and $\boldsymbol{\varphi}(\boldsymbol{s},t)=( \tilde{\varphi}\cos \theta^{1}, 0, 0)$, where $\tilde{\varphi} \in C^{\infty}_{0}(0,T)$, yields that for almost every time there holds
\begin{align*}
\frac{d}{dt} \int_{I_{1}} (\cos \theta^{1}, \sin \theta^{1}) ds  =\int_{I_{1}}  \partial_{t} \theta^{1} (-\sin \theta^{1}, \cos \theta^{1}) ds = G^{1} - (\vec{\lambda}-\vec{\mu}) \cdot A^{1} 
\end{align*}
where we use the notation employed in \eqref{defG}, \eqref{matA}. Similarly testing with $\boldsymbol{\varphi}(\boldsymbol{s},t)=(0,-\tilde{\varphi} \sin \theta^{2} ,0 )$ and   
$\boldsymbol{\varphi}(\boldsymbol{s},t)=(0, \tilde{\varphi}\cos \theta^{2} , 0)$, respectively 
$\boldsymbol{\varphi}(\boldsymbol{s},t)=(0,0,-\tilde{\varphi} \sin \theta^{3}  )$ and   
$\boldsymbol{\varphi}(\boldsymbol{s},t)=(0, 0,\tilde{\varphi}\cos \theta^{3} )$
 where $\tilde{\varphi} \in C^{\infty}_{0}(0,T)$ yields
 that for almost every time
\begin{align*}
\frac{d}{dt} \int_{I_{2}} (\cos \theta^{2}, \sin \theta^{2}) ds=\int_{I_{2}}  \partial_{t} \theta^{2} (-\sin \theta^{2}, \cos \theta^{2}) ds &= G^{2} + \vec{\lambda} \cdot A^{2},\\
\frac{d}{dt} \int_{I_{3}} (\cos \theta^{3}, \sin \theta^{3}) ds=\int_{I_{3}}  \partial_{t} \theta^{3} (-\sin \theta^{3}, \cos \theta^{3}) ds &= G^{3} -\vec{\mu} \cdot A^{3}. 
\end{align*}
Using \eqref{constraintlimit} we infer that for almost every time the system \eqref{e1bis}, \eqref{e2bis} holds for $\bt$. 
Inequality \eqref{eq:3.20} now follows from \eqref{boundLLzero}  and \eqref{eq:4.1bis}. 
\end{proof}

\begin{rem}\rm
Notice that Theorem~\ref{thm:existence} does not yield uniqueness of solutions.
To that end a deeper analysis would be needed (see for instance \cite[Lemma~3.20]{OPW18} for a similar issue in the case of a single evolving curve).
\end{rem}

\begin{rem}\rm
The method of proof of Theorem \ref{thm:existence} slightly differs from the one presented in \cite{OPW18} since in this paper we treat the Lagrange multipliers implicitly. This has the advantage that no restriction on the time  $T$ is necessary to show existence, and that the decrease of the energy follows directly. In particular, there is no need to analyze higher regularity properties of solutions as in \cite{OPW18}.  With the techniques presented here  \cite[Thm~1.1]{OPW18} can be generalized in the following sense: under the hypothesis of \cite[Thm~1.1]{OPW18} then a weak solution to $(P)$ can be defined for any time $T \in (0, +\infty)$.
\end{rem}

So far we assumed \eqref{cond1}. However,
as noticed above, the estimates on the Lagrange multipliers 
given in Theorem \ref{thm:3.2}  hold as long as we assume that two of the three curves have positive total curvature, 
that is, if the corresponding angle functions have positive oscillation. By this observation and by Corollary \ref{corpaola},
we provide a partial extension of Theorem \ref{thm:existence}.

\begin{proof}[Proof of Theorem~\ref{thm:existence2}] 
Recalling Corollary \ref{corpaola}, it follows from \eqref{eq:osc} that there exists $T>0$ such that 
\[
\min\left( \text{osc}_{\bar{I}_{j_1}} \, \bar{\theta}_n^{j_1}(t),
\text{osc}_{\bar{I}_{j_2}} \, \bar{\theta}_n^{j_2}(t)\right) \ge \frac c2,
\]
for all $t\in [0,T]$ and $n\in\mathbb N$. As a consequence the estimates on the Lagrange multipliers 
given in Theorem \ref{thm:3.2} still hold, and we can proceed exactly as in the proof of Theorem \ref{thm:existence}.

To show the final assertion, it is enough to observe that, if $T_{max}<+\infty$ and the oscillation of $\theta^{j_1}$ and 
of $\theta^{j_2}$
are uniformly bounded below by $\delta>0$ on $[0,T_{max}]$, then we can extend the solution on a time interval 
$[0,T']$ with $T'=T'(\delta)>T_{max}$.
\end{proof}

\begin{rem}\label{rem-triod}\rm
Note that the result in Theorem \ref{thm:existence2} can be extended 
to the case of a network of three curves with a single triple junction and three fixed endpoints.
We notice that for such network the second order evolution, expressed in terms of the functions $\theta^j$, is again given by the equations \eqref{E1}--\eqref{E3}.  
Moreover, the natural boundary conditions are still given by \eqref{E4}, whereas the condition \eqref{constraintlimit} becomes
\[
\int_{I_{1}} (\cos \theta^{1}, \sin \theta^{1}) ds - P_1= 
\int_{I_{2}} (\cos \theta^{2}, \sin \theta^{2}) ds- P_2=
\int_{I_{3}} (\cos \theta^{3}, \sin \theta^{3}) ds- P_3\,,
\]
where $P_1,P_2,P_3$ are the fixed endpoints.
\end{rem}


\subsection{Long-time behavior}\label{sec:long}

We now show that the weak solutions given by Theorem \ref{thm:existence} converge,
on a suitable sequence of times, to a critical point of the energy.

\begin{proof}[Proof of Theorem \ref{thm:longtime}]
{}From \eqref{eq:4.1} we know that, for $j=1,2,3$ we have
\begin{align}
\frac{1}{2}\int_{0}^{\infty} \int_{I_{j}} |\partial_{t} \theta^{j}|^2 ds dt  
\leq D(\bt_{0}).
\end{align}
Together with \eqref{eq:3.20} this yields the existence of a sequence of times $(t_{n})_{ n \in \N}$, and vectors $\vec{\lambda}, \vec{\mu} \in \R^{2}$ such that $t_{n} \to \infty$ and
\begin{align}\label{aiutino2}
\vec{\lambda}(t_{n}) \to \vec{\lambda}, \qquad \vec{\mu}(t_{n}) \to \vec{\mu}, \qquad \| \partial_{t} \theta^{j}(t_{n}) \|_{L^{2}(I_{j})} \to 0,
\end{align}
for $j=1,2,3$, as $n \to \infty$.
From \eqref{eq:4.1bis} we infer that for $j=1,2,3$ we have
$$ \| \partial_{s} \theta^{j}(t_{n}) \|_{L^{p}(I_{j})} \leq C(D(\bt_{0}),p).$$
Moreover, from
\begin{align*}
|\theta^{j}(s, t_{n})- \theta^{j}(0, t_{n})| \leq (L_{j})^{\frac{p-1}{p}} \| \partial_{s} \theta^{j}(t_{n}) \|_{L^{p}(I_{j})} \leq C
\end{align*}
 for any $s \in [0,L_{j}]$,
  we obtain that the sequence $\tilde{\theta}^{j}(\cdot, t_{n}):=\theta^{j}(\cdot, t_{n})- 2\pi z_{n}$, with $z_{n}\in \Z$ chosen in such a way that $ |\theta^{j}(0, t_{n}) - 2\pi z_{n}| \leq 2\pi$,
 satisfies in addition the uniform bound
 $$ \|  \tilde{\theta}^{j}(t_{n}) \|_{W^{1,p}(I_{j})} \leq C(D(\bt_{0}),p, L_{j}).$$
Therefore, by Arzel\`a-Ascoli Theorem, possibly extracting a further subsequence we have that $\tilde{\theta}^{j}(t_{n}) \rightarrow \theta^{j}_{\infty}$ uniformly as $n\to \infty$. 
Notice also that $\tilde{\theta}^{j}(t_{n}) \rightharpoonup \theta^{j}_{\infty}$ weakly in $W^{1,p}(I_{j})$,   
and that from the uniform bounds 
$$\| |\partial_{s} \tilde\theta^{j}(t_{n})|^{p-2}\partial_{s} \tilde\theta^{j}(t_{n})\|_{L^{\frac{p}{p-1}} (I_{j})} \leq C
\quad\text{ and }\quad\|( |\partial_{s} \tilde\theta^{j}(t_{n})|^{p-2}\partial_{s} \tilde\theta^{j}(t_{n}))_{s}\|_{L^{2}(I_{j})} \leq C$$ (which follows from \eqref{aiutino2}, \eqref{E1}, \eqref{E2}, \eqref{E3}) we infer  a uniform bound in the $H^{1}$-norm  giving that 
 $|\partial_{s} \tilde\theta^{j}(t_{n})|^{p-2}\partial_{s} \tilde\theta^{j}(t_{n})\rightharpoonup
|\partial_{s} \theta^{j}_\infty|^{p-2}\partial_{s} \theta^{j}_\infty$ weakly
in $H^{1}(I_{j})$  and uniformly on $I_{j}$.
Since, if $\bt$ solves \eqref{E1}, \eqref{E2}, \eqref{E3}, then so does $\bt-(2\pi z^{1}, 2\pi z^{2}, 2\pi z^{3})$ with $z^{i} \in \Z$ and with no change in the Lagrange multipliers, we can pass to the limit as $n\to \infty$ and obtain
 that $\bt_{\infty}:=(\theta_{\infty}^{1},\theta_{\infty}^{2},\theta_{\infty}^{3})$ satisfies the constraint \eqref{constraintlimit}.
 Moreover, passing to the limit in \eqref{eqtest} we infer that 
$\bt_{\infty}$ also fulfills the system  \eqref{sislim},
together with the boundary conditions \eqref{condlim}.
\end{proof}



\appendix\renewcommand{\thesection}{\Alph{section}}
\setcounter{equation}{0}
\renewcommand{\theequation}{\Alph{section}\arabic{equation}}

\section{Appendix}

\begin{proof}[Proof of Lemma~\ref{lem:4.1}] 
We adapt to the present setting the arguments presented in \cite[Lemma~3.11]{OPW18}.
Let $j \in \{1,2,3\}$, and let $q:=p/(p-1)$. 
First of all notice that by Theorem \ref{thm:3.2}   we have that $\theta^{j}_{n}(\cdot,t)  \in W^{1,p}(I_{j})$ 
for any $t \in [0,T]$, and
\begin{align} \label{eq:4.6old}
\sup_{t \in [\,0, T\,]}  \| \theta^{j}_{n}(\cdot,t) \|_{L^{2}(I_{j})}
 \le  C(T, D(\bt_{0}) , \| \theta^{j}_{0}\|_{L^{2}(I_{j})}), \qquad \sup_{t \in [\,0, T\,]}  \| \partial_{s}\theta^{j}_{n}(\cdot,t) \|^{p}_{L^{p}(I_{j})}
 \le  p \,D(\bt_{0}) .
\end{align} 
Therefore there exists a map $\theta^{j} \in L^{\infty}(0,T; W^{1,p}(I_{j}))$ such that, up to a subsequence, 
\begin{align*}
\theta^{j}_{n} \rightharpoonup \theta^{j} \quad \text{weakly$\ast$} \text{ in }  L^{\infty}(0,T; W^{1,p}(I_{j})) \qquad \text{as} \quad n \to \infty. 
\end{align*} 

Next note that, since $\theta^{j}_{n}(s, \cdot)$ is absolutely continuous in $[\,0, T\,]$, 
we infer from H\"older's inequality that  
\begin{align*}
\| \theta^{j}_{n}(\cdot, t_{2}) - \theta^{j}_{n}(\cdot, t_{1}) \|_{L^{2}(I_{j})} 
 &= \left( \int_{I_{j}} \left|\int^{t_{2}}_{t_{1}} \dfrac{\pd \theta^{j}_{n}}{\pd t}(s,\tau) \, d\tau \right|^{2}\, ds \right)^{\frac{1}{2}} 
 \\
 &\le \left( \int^{t_{2}}_{t_{1}} \int_{I_{j}} \left|\dfrac{\pd \theta^{j}_{n}}{\pd t}(s,\tau)\right|^{2}\, ds d \tau \right)^{\frac{1}{2}} 
 (t_{2}-t_{1})^{\frac{1}{2}},
\end{align*} 
for any $t_{1}$, $t_{2} \in [\,0, T\,]$ with $t_{1}< t_{2}$.
Using Theorem \ref{thm:3.2}, we find that
\begin{align} \label{eq:4.4}
\int^{t_{2}}_{t_{1}} \int_{I_{j}} \left| \dfrac{\pd \theta^{j}_{n}}{\pd t}(s,\tau)\right|^{2}\, ds d \tau
 \leq \int^{t_{2}}_{t_{1}} \int_{I_{j}} |\boldsymbol{V}_{n}(s,\tau)|^{2} \, ds d \tau 
 \le 2 D(\bt_{0}).  
\end{align}
Hence we obtain 
\begin{align} \label{eq:4.5}
\| \theta^{j}_{n}(\cdot, t_{2}) - \theta^{j}_{n}(\cdot, t_{1}) \|_{L^{2}(I_{j})} \le  \sqrt{2 D(\bt_{0}) } (t_{2}-t_{1})^{\frac{1}{2}}. 
\end{align}

We now turn to the proof of \eqref{eq:4.3}. 
First of all observe that by \eqref{eq:4.6old} and embedding theory we have that
\begin{align}\label{linftyb}
\sup_{ t \in [0,T]} \| \theta^{j}_{n} (\cdot, t) \|_{L^{\infty}(I_{j})} \leq C, \qquad \text{ with } C=C(L_{j},T,D(\bt_{0}),\| \theta^{j}_{0}\|_{L^{2}(I)},p).
\end{align}
Moreover, again by \eqref{eq:4.6old}, for any $t \in [0,T]$ we have
\begin{align}\label{pip}
|\theta^{j}_{n}(s_{2},t) - \theta^{j}_{n}(s_{1},t)| \leq 
\int_{s_{1}}^{s_{2}}
\left|   \partial_s \theta^{j}_{n}(s, t)\right| ds \leq C |s_{2}-s_{1}|^{1/q},
\end{align}
with $C=C(D(\bt_{0}),p)$.
Fix $0 \le t_1 \le t_2 \le T$ arbitrarily and set $$\Gm(\cdot):= \theta^{j}_n(\cdot, t_2) - \theta^{j}_n(\cdot, t_1) \in W^{1,p}(I_{j}).
$$ 
By an interpolation inequality (see for instance \cite[Thm. 5.9]{Adams})
we find  
$
\| \Gm \|_{L^\infty} \le C \| \Gm \|_{L^q}^{1/2} \|  \Gm \|_{W^{1,p}}^{1/2},  
$
and using \eqref{eq:4.6old} we get
$\| \Gm \|_{L^\infty} \le C \| \Gm \|_{L^q}^{1/2}$. In particular, for $p\ge 2$ we have 
$\| \Gm \|_{L^\infty} \le C(L_j) \| \Gm \|_{L^2}^{1/2}$,
so that by 
 \eqref{eq:4.5} we infer
\begin{equation}\label{lala}
\| \Gm \|_{L^\infty} \le C | t_2 - t_1 |^{1/4},
\end{equation}
with $C= C(T, L_j, D(\bt_{0}) , \| \theta^{j}_{0}\|_{L^{2}(I)},p)$. For $p\in (1,2)$, that is $q>2$, 
another interpolation inequality gives $\| \Gm \|_{L^q} \le \| \Gm \|_{L^2}^\theta \| \Gm \|_{L^\infty}^{1-\theta}$,
with $\theta=2/q$. Recalling that $\Gm\in L^\infty$, we then obtain
\begin{equation}\label{lulu}
\| \Gm \|_{L^\infty} \le C\| \Gm \|_{L^q}^\frac 12\le 
C\| \Gm \|_{L^2}^\frac \theta 2  \| \Gm \|_{L^\infty}^\frac{1-\theta}{2}  
\le C\| \Gm \|_{L^2}^\frac \theta 2 = C\| \Gm \|_{L^2}^\frac 1q
\le C(t_2-t_1)^\frac{p-1}{2p}.
\end{equation}
From \eqref{lala} and \eqref{lulu} it follows that 
\begin{align}  \label{eq:4.12}
\| \theta^{j}_{n}(t_{2}) - \theta^{j}_{n}(t_{1}) \|_{C^{0}(I_{j})} \le C | t_{2} - t_{1} |^{\alpha} . 
\end{align}
From the above inequality and \eqref{pip} we then get
\begin{align}
|\theta^{j}_{n}( s_{2}, t_{2}) - \theta^{j}_{n}(s_{1}, t_{1})| \leq C (|t_{2}-t_{1}|^{\alpha} + |s_{2}-s_{1}|^{1/q})
\le  C (|t_{2}-t_{1}|^{\alpha} + |s_{2}-s_{1}|^{\alpha})
\end{align}
for any $(t_{i}, s_{i}) \in [0,T] \times [0,L_{j}]$, $i=1,2$.
Application of the Arzel\`a-Ascoli Theorem yields \eqref{eq:4.3}. 
In particular $\theta^{j} ( \cdot, t) \in W^{1,p}(I_{j})$ for all times $t\in [0,T]$.
Moreover, setting $t_{1}=0$ in \eqref{eq:4.12}, we have that 
\begin{align*}
\| \theta^{j}(t) - \theta^{j}_{0} \|_{C^{0}(I_{j})} \to 0 \quad \text{as} \quad t \downarrow 0. 
\end{align*}  
From \eqref{eq:4.4}  we also infer that  there exists $V^{j} \in L^{2}(0,T; L^{2}(I_{j})) $ such that
\begin{align}\label{VH}
V^{j}_{n} = \pd_{t} \theta^{j}_{n} \rightharpoonup V^{j}  \quad \text{in} \quad L^{2}(0,T; L^{2}(I_{j})). 
\end{align}
Moreover, for any $v \in C_{0}^{\infty}((0,T) \times I_{j})$, we have
\begin{align*}
\int_{0}^{T} \int_{I_{j}} \theta^{j}_{n} v_{t}\,  ds dt = -\int_{0}^{T}\int_{I_{j}}  V^{j}_{n}  v  \, dsdt \to  -\int_{0}^{T} \int_{I_j} V^{j} v \, dsdt
\end{align*}
as $n\to\infty$, and using the fact  that $\theta^{j}_{n} \to \theta^{j}$ uniformly, we obtain
\begin{align*}
\int_{0}^{T}\int_{I_{j}} \theta^{j}_{n} v_{t} \, ds dt \to \int_{0}^{T} \int_{I_{j}} \theta^{j} v_{t} \, ds dt,
\end{align*}
from which we infer that $\theta^{j}$ admits weak derivative $\theta^{j}_{t}=V^{j}$,
$\theta^{j} \in H^{1}(0,T; L^{2}(I_{j}))$ and \eqref{eq:4.2} holds. 
Finally, it follows from \eqref{eq:4.4} that \eqref{eq:4.1} also holds. 
\end{proof}

\begin{proof}[Proof of Lemma~\ref{lem:4.2}]
This is a straight-forward adaptation of \cite[Lemma~3.12]{OPW18}.
Let $j\in \{ 1,2,3 \}$.
We show the proof only for $\bar{\theta}^{j}_{n}$, since analogous arguments holds for $\underline{\theta}^{j}_{n}$.
Recalling \eqref{eq:4.6old}, we  see that $\bar{\theta}^{j}_{n} \in L^{\infty}(0,T; W^{1,p}(I_{j}))$. 
In particular, equations \eqref{linftyb}, \eqref{pip}  hold with $\theta^{j}_{n}$ replaced by $\bar{\theta}^{j}_{n}$. 
 
Fix now $t \in (\,0, T\,]$ arbitrarily. 
Then there exists a family of intervals $\{ (\,(i_{n}-1) \tau_{n}, i_{n} \tau_{n} \,] \}_{n \in \N}$ such that 
$t \in (\,(i_{n}-1) \tau_{n}, i_{n} \tau_{n} \,]$. 
From \eqref{eq:4.12} we infer  that 
\begin{align*}
\| \bar{\theta}^{j}_{n}(t) - \theta^{j}_{n}(t) \|_{C^{0}(I_{j})} 
& = \| \theta^{j}_{i_{n},n} - \theta^{j}_{n}(t) \|_{C^{0}(I_{j})} 
 = \| \theta^{j}_{n}(i_{n} \tau_{n}) - \theta^{j}_{n}(t) \|_{C^{0}(I_{j})} \\
& \le C |i_{n} \tau_{n} - t|^{\alpha} 
 \le C \tau_{n}^{\alpha} \to 0 \quad \text{as} \quad n \to \infty  .
\end{align*} 
Since $\theta^{j}_{n} \to \theta^{j}$ in $C^{0}([0,T] \times I_{j})$ by  Lemma~\ref{lem:4.1}, and $t$ was arbitrarily chosen, we infer that $\bar{\theta}^{j}_{n} \to \theta^{j}$ in $C^{0}([0,T] \times I_{j})$. 

We turn to the proof of \eqref{eq:4.14}. 
Recalling again \eqref{eq:4.6old}, we also see that $\partial_{s}\bar{\theta}^{j}_{n} \in L^{p}(0,T; L^{p}(I_{j}))$ and $\| \partial_{s}\bar{\theta}^{j}_{n}\|_{L^{p}(0,T; L^{p}(I_{j}))} \leq C$, for all $n \in \mathbb{N}$. Since $L^{p}(0,T; L^{p}(I_{j}))$ is a reflexive Banach space  there exists $ v^{j} \in L^{p}(0,T; L^{p}(I_{j}))$ such that $ \partial_{s}\bar{\theta}^{j}_{n} \rightharpoonup  v^{j}$.
This implies that 
$$ \int^{T}_{0} \int_{I_{j}} \partial_{s}\bar{\theta}^{j}_{n} \cdot \vp  \, dsdt \to  \int^{T}_{0} \int_{I_{j}}  v^{j} \cdot \vp \, dsdt $$
for any $\vp \in L^{q}(0,T; L^{q}(I_{j}))$, with $q=p/(p-1)$.
On the other hand, if $\vp \in L^{\infty}(0,T; C_{0}^{\infty}(I_{j}))$ we infer that
$$  \int^{T}_{0} \int_{I_{j}}  \partial_{s}\bar{\theta}^{j}_{n} \cdot \vp \,dsdt = - \int^{T}_{0} \int_{I_{j}}  \bar{\theta}^{j}_{n} \cdot \partial_{s}\vp \, dsdt \to 
 - \int^{T}_{0} \int_{I_{j}}  \theta^{j} \cdot \partial_{s}\vp \, dsdt, $$
where we have used that $\bar{\theta}^{j}_{n} \to \theta^{j}$. Hence we obtain that $v^{j}=\partial_{s}\theta^{j}$, 
and the claim follows.
\end{proof}

\begin{proof}[Proof of Lemma~\ref{lem:4.4}]
This is a straight-forward adaptation of \cite[Lemma~3.13]{OPW18}.
Notice first that, from \eqref{regularity}, \eqref{boundLLdis}, and Theorem~\ref{thm:3.2} it follows that
\begin{align}\label{ester}
\| |\partial_s\theta^{j}_{i,n} |^{p-2} \partial_s \theta^{j}_{i,n} \|_{H^1 (I_{j})} &  \leq C (1 +\sum_{r=1}^{3}\| V_{i,n}^{r} \|_{L^{2}(I_{j})} ) 
\end{align}
for $j=1,2,3$ and for all $i=1, \ldots,n$, where $C=C(L_{j},p, D(\bt_{0}))$. 
Recalling Theorem~\ref{thm:3.2}, for all $j=1,2,3$ we get 
\begin{align}
\int^{T}_{0} \|  | \partial_{s}\bar{\theta}^{j}_{n}|^{p-2}  \partial_{s}\bar{\theta}^{j}_{n} \|_{H^1(I_j)}^{2} \, dt 
 \le C \int^{T}_{0} ( 1 + \| \V_{n} \|^{2}_{L^{2}(I_j)} ) \, dt 
 \le C. 
\end{align}
Thus, recalling also \eqref{nbc99}, we find $w^j \in L^{2}(0,T; H^1_{0}(I_j))$ such that 
\begin{align} \label{eq:4.18}
 | \partial_{s}\bar{\theta}^{j}_{n}|^{p-2}  \partial_{s}\bar{\theta}^{j}_{n} \rightharpoonup w^j \quad \text{in} \quad 
L^{2}(0,T; H^1(I_j)) \quad \text{as} \quad n \to \infty ,
\end{align}
 and
$\| w^j \|_{L^{2}(0,T; H^1(I_j))} \leq C$. In particular,
this implies that
\begin{align}\label{peppa}
\int_{0}^{T} \int_{I_j}  | \partial_{s}\bar{\theta}^{j}_{n}|^{p-2}  \partial_{s}\bar{\theta}^{j}_{n} 
 \cdot \vp \, ds dt \to \int_{0}^{T} \int_{I_j} w^{j} \cdot \vp \, ds dt ,
\end{align}
for all $ \vp\in L^{2}(0,T; L^{2}(I_j))$.
On the other hand, letting $q:=p/(p-1)$, from \eqref{ester} it follows  that
\[
\| |(\theta^{j}_{i,n} )_{s}|^{p-2} (\theta^{j}_{i,n} )_{s} \|_{L^{q} (I_{j})}
\leq C \| |(\theta^{j}_{i,n} )_{s}|^{p-2} (\theta^{j}_{i,n} )_{s}\|_{H^1 (I_{j})}
\leq C (1 +\sum_{r=1}^{3}\| V_{i,n}^{r} \|_{L^{2}(I_{j})} ) ,
\]
so that,
by Theorem~\ref{thm:3.2} we also get
\begin{align*}
\int^{T}_{0} \|   | \partial_{s}\bar{\theta}^{j}_{n}|^{p-2}  \partial_{s}\bar{\theta}^{j}_{n}   \|_{L^{q}(I_j)}^{2} \, dt 
 \le C \int^{T}_{0} ( 1 + \| \V_{n} \|^{2}_{L^{2}(I_j)} ) \, dt 
 \le C.
\end{align*}
The space $L^{2}(0,T; L^{q}(I_j))$ is reflexive with dual space given by $(L^{2}(0,T; L^{q}(I_j)))^{*}= L^{2}(0,T; L^{p}(I_j))$. Hence there exists $\tilde{\xi}^j \in L^{2}(0,T; L^{q}(I_j))$ such that
\begin{align}\label{peppa2}
\int_{0}^{T} \int_{I_j}  | \partial_{s}\bar{\theta}^{j}_{n}|^{p-2}  \partial_{s}\bar{\theta}^{j}_{n} \cdot \vp \, ds dt \to 
\int_{0}^{T} \int_{I_j} \tilde{\xi}^j \cdot \vp \, ds dt  \quad \forall \vp \in L^{2}(0,T; L^{p}(I_j)).
\end{align}
Together with \eqref{peppa} and Lemma \ref{lem:4.2}, we infer that 
$w^j=\tilde{\xi}^j$. 


Next, we set 
\begin{align*}
F(\psi) := \frac{1}{p}\,\| \psi_{s} \|_{L^{p}(0,T; L^{p}(I_j))}^{p}.  
\end{align*}
Using the convexity of the map $y \to \frac{1}{p} |y|^{p}$, we see that 
\begin{align} \label{eq:4.19}
F(\psi) - F(\bar{\theta}^{j}_{n})  
 \ge \int^{T}_{0} \int_{I_j} | \partial_{s}\bar{\theta}^{j}_{n}|^{p-2}  \partial_{s}\bar{\theta}^{j}_{n}
 \cdot (\psi - \bar{\theta}^{j}_{n})_{s} \, ds dt 
 \quad \text{for any} \  \psi \in L^{\infty}(0,T; W^{1,p}(I_j)). 
\end{align}
Recalling \eqref{eq:4.14} 
and letting $n \to \infty$ in~\eqref{eq:4.19}, we have   
\begin{align} \label{eq:4.20}
F(\psi) - F(\theta^j) \ge \int^{T}_{0} \int_{I_j} w^j \cdot (\psi - \theta^j)_{s} \, ds dt, 
\end{align}
where we have used integration by parts (recall \eqref{nbc99} and $w^{j} \in L^{2}(0,T; H^{1}_{0}(I_{j}))$) and \eqref{eq:4.13}.
Letting now $\psi= \theta^j + \ve \vp$ in \eqref{eq:4.20}  for some $\vp \in L^{\infty}(0,T; W^{1,p}(I_j))$ and $\ve >0$, we obtain 
\begin{align} \label{eq:4.21}
\dfrac{F(\theta^j + \ve \vp) - F(\theta^j)}{\ve} \ge \int^{T}_{0} \int_{I_j} w^j \cdot  \vp_{s} \, ds dt.
\end{align}
On the other hand, letting $\psi= w^j - \ve \vp$ in \eqref{eq:4.20}, we also have
\begin{align} \label{eq:4.22}
\dfrac{F(\theta^j) - F(\theta^j - \ve \vp)}{\ve} \le \int^{T}_{0} \int_{I_j} w^j \cdot  \vp_{s} \, ds dt.
\end{align}
Letting $\ve \downarrow 0$, from \eqref{eq:4.21} and \eqref{eq:4.22} 
we then get
\begin{align}\label{peppa4}
\int^{T}_{0} \int_{I_j}  | \partial_{s}\theta^j|^{p-2} \theta^j_s \cdot \vp_{s} \, ds dt 
 = \int^{T}_{0} \int_{I_j} w^j \cdot  \vp_{s} \, ds dt = \int^{T}_{0} \int_{I_j} \tilde\xi^j \cdot  \vp_{s} \, ds dt,
\end{align}
for all $\vp \in L^{\infty}(0,T; W^{1,p}(I_j))$. 
Together with \eqref{peppa2} this gives the thesis.
\end{proof}

\begin{lem}\label{corPaola2}
Suppose $\theta=\theta(s,t) \in C^{0, \alpha}([0,T] \times \bar{I})$ for some $\alpha \in (0,1]$. Then $\text{osc } \theta \in C^{\alpha}([0,T])$.
\end{lem}
\begin{proof}
By definition we have
$\text{osc } \theta(t) = \max_{s\in \bar{I}} \theta(s,t)  - \min_{s\in \bar{I}} \theta(s,t)=: \theta (\bar{s}, t) - \theta (\underline{s}, t). $ Using this notation it follows
\begin{align*}
\text{osc } \theta(t_{1}) - \text{osc } \theta(t_{2}) &=  [\theta (\bar{s}_{1}, t_{1}) - \theta (\underline{s}_{1}, t_{1}) ] - [\theta (\bar{s}_{2}, t_{2}) - \theta (\underline{s}_{2}, t_{2}) ]\\
& \leq \theta (\bar{s}_{1}, t_{1}) - \theta (\bar{s}_{1}, t_{2}) + \theta (\underline{s}_{1}, t_{2}) -\theta (\underline{s}_{1}, t_{1}) \leq C |t_{1} - t_{2}|^{\alpha},
\end{align*}
where we have used $\theta (\bar{s}_{2}, t_{2}) \geq \theta (\bar{s}_{1}, t_{2})$  and $\theta (\underline{s}_{2}, t_{2}) \leq \theta (\underline{s}_{1}, t_{2})$. The claim follows.
\end{proof}




\bibliography{ref}

\begin{thebibliography}{10}

\bibitem{Adams}
{\sc Adams, R.~A., and Fournier, J. J.~F.}
\newblock {\em Sobolev spaces}, second~ed., vol.~140 of {\em Pure and Applied
  Mathematics}.
\newblock Elsevier/Academic Press, Amsterdam, 2003.

\bibitem{ATW}
{\sc Almgren, F., Taylor, J.~E., and Wang, L.}
\newblock Curvature-driven flows: a variational approach.
\newblock {\em SIAM J. Control Optim. 31}, 2 (1993), 387--438.

\bibitem{Ba19}
{\sc Badal, R.}
\newblock Curve-shortening flow of open, elastic curves in $\mathbb{R}^2$ with
  repelling endpoints: A minimizing movement approach.
\newblock {\em Preprint\/} (2019).

\bibitem{BK}
{\sc Bellettini, G., and Kholmatov, S.}
\newblock Minimizing movements for mean curvature flow of partitions.
\newblock {\em Preprint\/} (2017).

\bibitem{BR}
{\sc Bronsard, L., and Reitich, F.}
\newblock On three-phase boundary motion and the singular limit of a
  vector-valued {G}inzburg-{L}andau equation.
\newblock {\em Arch. Rational Mech. Anal. 124}, 4 (1993), 355--379.

\bibitem{DCP14}
{\sc Dall'Acqua, A., Lin, C.-C., and Pozzi, P.}
\newblock Evolution of open elastic curves in {$\Bbb{R}^n$} subject to fixed
  length and natural boundary conditions.
\newblock {\em Analysis (Berlin) 34}, 2 (2014), 209--222.

\bibitem{DCP17}
{\sc Dall'Acqua, A., Lin, C.-C., and Pozzi, P.}
\newblock A gradient flow for open elastic curves with fixed length and clamped
  ends.
\newblock {\em Ann. Sc. Norm. Super. Pisa Cl. Sci. (5) 17}, 3 (2017),
  1031--1066.

\bibitem{DLP18}
{\sc Dall'Acqua, A., Lin, C.-C., and Pozzi, P.}
\newblock Flow of elastic networks: long-time existence result.
\newblock {\em Preprint\/} (2018).

\bibitem{DNP19}
{\sc Dall'Acqua, A., Novaga, M., and Pluda, A.}
\newblock Minimal elastic networks.
\newblock {\em Indiana Univ. Math. J\/} (to appear).

\bibitem{DP14}
{\sc Dall'Acqua, A., and Pozzi, P.}
\newblock A {W}illmore-{H}elfrich {$L^2$}-flow of curves with natural boundary
  conditions.
\newblock {\em Comm. Anal. Geom. 22}, 4 (2014), 617--669.

\bibitem{DPS}
{\sc Dall'Acqua, A., Pozzi, P., and Spener, A.}
\newblock The {\l} ojasiewicz-{S}imon gradient inequality for open elastic
  curves.
\newblock {\em J. Differential Equations 261}, 3 (2016), 2168--2209.

\bibitem{DKS}
{\sc Dziuk, G., Kuwert, E., and Sch{\"a}tzle, R.}
\newblock Evolution of elastic curves in {$\Bbb R^n$}: existence and
  computation.
\newblock {\em SIAM J. Math. Anal. 33}, 5 (2002), 1228--1245 (electronic).

\bibitem{Fusco}
{\sc Fonseca, I., Fusco, N., Leoni, G., and Morini, M.}
\newblock Motion of elastic thin films by anisotropic surface diffusion with
  curvature regularization.
\newblock {\em Arch. Ration. Mech. Anal. 205}, 2 (2012), 425--466.

\bibitem{GMP19}
{\sc Garcke, H., Menzel, J., and Pluda, A.}
\newblock Long time existence of solutions to an elastic flow of networks.
\newblock {\em Preprint\/} (2019).

\bibitem{GMP18}
{\sc Garcke, H., Menzel, J., and Pluda, A.}
\newblock Willmore flow of planar networks.
\newblock {\em J. Differential Equations 266}, 4 (2019), 2019--2051.

\bibitem{Gu}
{\sc Gurtin, M.~E.}
\newblock {\em Thermomechanics of evolving phase boundaries in the plane}.
\newblock Oxford Mathematical Monographs. The Clarendon Press, Oxford
  University Press, New York, 1993.

\bibitem{KL}
{\sc Kinderlehrer, D., and Liu, C.}
\newblock Evolution of grain boundaries.
\newblock {\em Math. Models Methods Appl. Sci. 11}, 4 (2001), 713--729.

\bibitem{Koiso}
{\sc Koiso, N.}
\newblock On the motion of a curve towards elastica.
\newblock In {\em Actes de la {T}able {R}onde de {G}\'eom\'etrie
  {D}iff\'erentielle ({L}uminy, 1992)}, vol.~1 of {\em S\'emin. Congr.} Soc.
  Math. France, Paris, 1996, pp.~403--436.

\bibitem{LangerSinger1}
{\sc Langer, J., and Singer, D.~A.}
\newblock Curve straightening and a minimax argument for closed elastic curves.
\newblock {\em Topology 24}, 1 (1985), 75--88.

\bibitem{Lin12}
{\sc Lin, C.-C.}
\newblock {$L^2$}-flow of elastic curves with clamped boundary conditions.
\newblock {\em J. Differential Equations 252}, 12 (2012), 6414--6428.

\bibitem{LinKai18}
{\sc Lin, C.-C., and Lue, Y.-K.}
\newblock Evolving inextensible and elastic curves with clamped ends under the
  second-order evolution equation in {$\Bbb R^2$}.
\newblock {\em Geom. Flows 3}, 1 (2018), 14--18.

\bibitem{Lin15}
{\sc Lin, C.-C., Lue, Y.-K., and Schwetlick, H.~R.}
\newblock The second-order {$L^2$}-flow of inextensible elastic curves with
  hinged ends in the plane.
\newblock {\em J. Elasticity 119}, 1-2 (2015), 263--291.

\bibitem{LS}
{\sc Luckhaus, S., and Sturzenhecker, T.}
\newblock Implicit time discretization for the mean curvature flow equation.
\newblock {\em Calc. Var. Partial Differential Equations 3}, 2 (1995),
  253--271.

\bibitem{MMN}
{\sc Magni, A., Mantegazza, C., and Novaga, M.}
\newblock Motion by curvature of planar networks, {II}.
\newblock {\em Ann. Sc. Norm. Super. Pisa Cl. Sci. (5) 15\/} (2016), 117--144.

\bibitem{MNP}
{\sc Mantegazza, C., Novaga, M., and Pluda, A.}
\newblock Motion by curvature of networks with two triple junctions.
\newblock {\em Geom. Flows 2\/} (2017), 18--48.

\bibitem{MNPS16}
{\sc Mantegazza, C., Novaga, M., Pluda, A., and Schulze, F.}
\newblock Evolution of networks with multiple junctions.
\newblock {\em Preprint\/} (2016).

\bibitem{MNT}
{\sc Mantegazza, C., Novaga, M., and Tortorelli, V.~M.}
\newblock Motion by curvature of planar networks.
\newblock {\em Ann. Sc. Norm. Super. Pisa Cl. Sci. (5) 3}, 2 (2004), 235--324.

\bibitem{NO15}
{\sc Novaga, M., and Okabe, S.}
\newblock Regularity of the obstacle problem for the parabolic biharmonic
  equation.
\newblock {\em Math. Ann. 363}, 3-4 (2015), 1147--1186.

\bibitem{NO17}
{\sc Novaga, M., and Okabe, S.}
\newblock Convergence to equilibrium of gradient flows defined on planar
  curves.
\newblock {\em J. Reine Angew. Math. 733\/} (2017), 87--119.

\bibitem{Oelz14}
{\sc \"Oelz, D.}
\newblock Convergence of the penalty method applied to a constrained curve
  straightening flow.
\newblock {\em Commun. Math. Sci. 12}, 4 (2014), 601--621.

\bibitem{Oelz11}
{\sc \"Oelz, D.~B.}
\newblock On the curve straightening flow of inextensible, open, planar curves.
\newblock {\em SeMA J.}, 54 (2011), 5--24.

\bibitem{Okabe2007}
{\sc Okabe, S.}
\newblock The motion of elastic planar closed curves under the area-preserving
  condition.
\newblock {\em Indiana Univ. Math. J. 56}, 4 (2007), 1871--1912.

\bibitem{Okabe2008}
{\sc Okabe, S.}
\newblock The dynamics of elastic closed curves under uniform high pressure.
\newblock {\em Calc. Var. 33}, 4 (2008), 493--521.

\bibitem{OPW18}
{\sc Okabe, S., Pozzi, P., and Wheeler, G.}
\newblock A gradient flow for the p-elastic energy defined on closed planar
  curves.
\newblock {\em Preprint\/} (2018).

\bibitem{Po}
{\sc Polden, A.}
\newblock Curves and surfaces of least total curvature and fourth-order flows.
\newblock {\em PhD Thesis, Universit\"at T\"ubingen\/} (1996).

\bibitem{Taylor}
{\sc Taylor, J.~E.}
\newblock Motion of curves by crystalline curvature, including triple junctions
  and boundary points.
\newblock In {\em Differential geometry: partial differential equations on
  manifolds ({L}os {A}ngeles, {CA}, 1990)}, vol.~54 of {\em Proc. Sympos. Pure
  Math.} Amer. Math. Soc., Providence, RI, 1993, pp.~417--438.

\bibitem{Wen}
{\sc Wen, Y.}
\newblock {$L^2$} flow of curve straightening in the plane.
\newblock {\em Duke Math. J. 70}, 3 (1993), 683--698.

\bibitem{wen95}
{\sc Wen, Y.}
\newblock Curve straightening flow deforms closed plane curves with nonzero
  rotation number to circles.
\newblock {\em J. Differential Equations 120}, 1 (1995), 89--107.

\bibitem{Wheeler}
{\sc Wheeler, G.}
\newblock Global analysis of the generalised {H}elfrich flow of closed curves
  immersed in {$\Bbb{R}^n$}.
\newblock {\em Trans. Amer. Math. Soc. 367}, 4 (2015), 2263--2300.

\end{thebibliography}
\bibliographystyle{acm}

\end{document}